\documentclass[10pt]{article}
 
\usepackage[margin=1in]{geometry} 
\usepackage{amsmath,amsthm,amssymb,mathtools}
\usepackage{color}
\usepackage[dvipsnames]{xcolor}
\usepackage{tikz}
\usepackage{lipsum}
\usepackage{fancyhdr}
\usepackage{hyperref}
\usepackage{cleveref}
\usepackage{diagbox}
\usepackage{float}
\usepackage{geometry,pdflscape}
\usepackage{thm-restate}
\newtheorem{theorem}{Theorem}[section]

\newtheorem{proposition}[theorem]{Proposition}

\newtheorem{lemma}[theorem]{Lemma}
\newtheorem{definition}[theorem]{Definition}

\newcommand{\C}[0]{\mathbb{C}}
\newcommand{\R}[0]{\mathbb{R}}

\newcommand{\speclin}[0]{\mathfrak{sl}_{r+1}(\mathbb{C})}

\newcommand{\hr}[0]{\Tilde{\alpha}}

\newcommand{\newmu}[0]{\alpha_I}

\newcommand{\thealtset}[0]{\mathcal{A}(\hr,\newmu)}
\newcommand{\qpartition}[1]{\wp_q\left(#1\right)}

\begin{document}
\title{
On the $q$-multiplicity of sums of distinct simple roots of $\mathfrak{sl}_{r+1}(\mathbb{C})$
}
\author{Matt McClinton}
\maketitle
\begin{abstract}
    In combinatorial representation theory, Kostant's weight multiplicity formula $m(\lambda,\mu)$ is a tool that provides a means of determining the multiplicity of a weight $\mu$ in the adjoint representation of a simple Lie algebra $\mathfrak{g}$, and in this work we consider the case of $\mathfrak{g}=\mathfrak{sl}_{r+1}(\mathbb{C})$. 
    In practice, performing calculations of Kostant's weight multiplicity formula is computationally intense, as the number of terms in this alternating sum grows factorially as the rank $r$ increases, of which most terms provide zero contribution to the overall sum. 
    In this work, we determine the Weyl alternation set, that is the terms in the alternating sum with nonzero contribution, for integral weights $\lambda$ the highest root of $\mathfrak{sl}_{r+1}(\mathbb{C})$, and $\mu$ any nonempty collection of distinct simple roots. We show that the alternation set is enumerated by a product of Fibonacci numbers, with the product being dependent on the choice of distinct simple roots. Then we compute the weight $q$-multiplicity for any nonempty collection of distinct simple roots.
\end{abstract}
\section{Introduction}
Let $[n]=\{1,2,\ldots,n\}$. Recall a Lie algebra is a vector space $\mathfrak{g}$ paired with a Lie bracket $[\cdot,\cdot]$ that is both skew symmetric and satisfies the Jacobi identity \cite{Goodman_Wallach_2009}. We consider the special linear Lie algebra 
\[\mathfrak{g}=\speclin=\{X\in M_{r+1}(\C):tr(X)=0\},\]
with Lie bracket $[X,Y]=XY-YX$ for any $X,Y\in \speclin$. 
For the standard basis vectors of $\R^{r+1}$, denoted  $e_1,e_2,\ldots,e_{r+1}$, we define $\alpha_i=e_i-e_{i+1}$ and refer to $\alpha_i$ as a simple root. We let $\Delta=\{\alpha_1,\alpha_2,\ldots,\alpha_r\}$ denote the set of simple roots and $\Phi^+=\Delta\cup \{\alpha_i+\alpha_{i+1}+\cdots+\alpha_j:1\leq i < j \leq r\}$ as the set of positive roots. We denote the positive roots for convenience as $\alpha_{i,j}=\alpha_i+\alpha_{i+1}+\cdots +\alpha_{j}$, with $\alpha_{i,i}=\alpha_i$. Denote $\hr=\alpha_{1,r}=\alpha_1+\alpha_2+\cdots +\alpha_r$, which we refer to as the highest root of $\speclin$. Additionally, let $\rho=\frac{1}{2}\sum_{\alpha\in \Phi^+}\alpha$ be one half the sum of the positive roots.

For each simple root $\alpha_i$, we pair it with a hyperplane $s_i$ orthogonal to $\alpha_i$. The Weyl group $W$ in rank $r$ is the group generated by the reflections through these hyperplanes $s_i$. It is known that for $\speclin$, the Weyl group in rank $r$ is isomorphic to symmetric group $\mathfrak{S}_{r+1}$. For any $\sigma\in W$, we denote $\ell(\sigma)=k$ as the length of $\sigma$, and specifically that $k$ is the smallest nonnegative integer such that $\sigma$ can be expressed as, possibly not uniquely, a product of $k$ reflections.
For simplicity, we assume throughout that any $\sigma\in W$ is stated as a reduced word expression.

In combinatorial representation theory, an area of active research 
is to determine the multiplicity of a weight $\mu$, a linear combination of simple roots with integer coefficients, in a highest weight representation of $\speclin$ with highest weight $\lambda$, see \cite{HarrisThesis, harry2024computingqmultiplicitypositiveroots, Harris_Lescinsky_Mabie_2018,visualizing}. A method of computing this multiplicity utilizes Kostant's weight multiplicity formula.
For weights $\lambda$ and $\mu$, \textbf{Kostant's weight multiplicity formula} is 
\begin{equation}
m(\lambda,\mu)=\sum_{\sigma\in W} (-1)^{\ell(\sigma)}\wp(\sigma(\hr+\rho)-\rho-\mu),\notag
\end{equation}
    where $\wp(\xi)$ is \textbf{Kostant's partition function}, which counts the number of ways to write $\xi$ as a nonnegative integral sum of positive roots \cite{Goodman_Wallach_2009}.
Note that for $\xi=c_1\alpha_1+c_2\alpha_2+\cdots+c_r\alpha_r$, we set $\wp(\xi)=0$ if any $c_i <0$, and we say $\wp(\xi)>0$ provided $c_i>0$ for every $i\in[r]$. 
Every element of the Weyl group corresponds to a term in the weight multiplicity. However, it is common for many terms to contribute a zero term to the sum. Since the Weyl group grows factorially as $r$ increases, the complexity of computing Kostant's weight multiplicity can be reduced significantly by first determining the terms that provide a nonzero contribution.
Given integral weights $\lambda$ and $\mu$, the \textbf{Weyl alternation set} is the set of elements in the Weyl group that correspond to terms in the weight multiplicity that provide nonzero contribution to the overall sum. We state it formally now:
    \begin{equation}\label{eq: defn of alternation set}
        \mathcal{A}(\lambda,\mu)=\{\sigma\in W:\wp(\sigma(\lambda+\rho)-\rho-\mu)>0\}.\notag
    \end{equation} 

In a fashion akin to generating functions, Lustzig \cite{Lustzig}
introduces a variant to Kostant's partition function and Kostant's weight multiplicity formula.
The \textbf{$q$-analog of Kostant's partition function} is the polynomial valued function 
    \begin{equation}
        \wp_q(\xi)=c_0+c_1q^1+\cdots+c_nq^n,\notag
    \end{equation}
    where $c_i$ equals the number of ways to write $\xi$ as a sum of exactly $i$ positive roots. The \textbf{$q$-analog of Kostant's weight multiplicity formula} is given by 
    \begin{equation}\label{eq: defn of q-analog}
        m_q(\lambda,\mu)=\sum_{\sigma\in W}(-1)^{\ell(\sigma)}\wp_q(\sigma(\hr+\rho)-\rho-\mu).\notag
    \end{equation}
Observe that if $\wp_q(\xi)=c_0+c_1q^1+\cdots+c_nq^n$, by setting $q=1$, we recover the original partition function $\wp(\xi)$ and, therefore, also recover $m(\lambda,\mu)$ from $m_q(\lambda,\mu)$.

Determining Weyl alternation sets and weight $q$-multiplicities for certain pairs of weights is an active area of research. We provide some results in the literature that we utilize throughout. 
Harris showed in \cite{HarrisThesis} that in rank $r$, when $\lambda$ is the highest root and $\mu=0$, then 
\begin{align}
    \mathcal{A}(\hr,0) = \left\{\sigma\in W: \sigma=\prod_{i\in\mathcal{I}}s_i\right\},\label{pamela alternation result}
\end{align}
where $\mathcal{I}$ is a (possibly empty) set of nonconsecutive integers in the interval $[2,r-1]$. Harris also showed that the weight $q$-multiplicity $m_q(\hr,0)$ reduces to a sum of powers of $q$:
 
\begin{equation}
    m_q(\hr,0)=\sum_{t=1}^r q^t.\label{eq: pamela result}
\end{equation}
Harry showed in \cite{harry2024computingqmultiplicitypositiveroots} that in rank $r$, when $\lambda$ the highest root and $\mu$ a positive root $\alpha_{i,j}$ with $1\leq i\leq j \leq r$, then 
\begin{equation}
    \mathcal{A}(\hr,\alpha_{i,j}) = \left\{\sigma\in W: \sigma=\prod_{i\in \mathcal{J}}s_i\right\}\label{kim alternation result},
\end{equation}
where $\mathcal{J}$ is a (possibly empty) collection of nonconsecutive integers in $[2,i-1]\sqcup [j+1,r-1]$. Additionally, Harry showed that the weight $q$-multiplicity reduces to a power of $q$, which is dependent on the choice of positive root $\alpha_{i,j}$: 
 
\begin{equation}
    m_q(\hr,\alpha_{i,j})=q^{r-j+i-1}.\label{kim result}
\end{equation}

Our main results extend the work of Harris and Harry, and answer a question of Paul Terwilliger, who asked what is the $q$-multiplicity when $\mu$  is a sum of distinct simple roots\footnote{Paul Terwilliger posed this question to Alex Wilson, who shared the question with the author.}. Throughout, we fix $\lambda=\hr=\alpha_1+\alpha_2+\cdots +\alpha_r$. 

Let $I\subseteq [r]$ be nonempty. 
We define the \textit{interval partition} of $I$ to be the partition of $I$ formed by taking the disjoint union of maximal length intervals of consecutive integers, which we define formally now.
\begin{definition}[Interval partition of the indexing set $I$]\label{def: interval partition}
    Let $I$ be any nonempty subset of $[r]$.  The \textbf{interval partition} of $I$, denoted as $\bigsqcup_{x=1}^{n(I)}[i_x,j_x]$, is a disjoint union of $n(I)$ intervals of consecutive integers, where each interval is selected to have maximal length.
    We impose the conditions that $1\leq i_1 \leq j_1 < i_2 \leq j_2 < \cdots < i_{n(I)} \leq j_{n(I)} \leq r$, and, for each $x\in [n(I)-1]$, we have $j_x\leq i_{x+1}+2$.     
    Furthermore, as usual, $|I|$ denotes the number of integers contained in $I$, and we denote the complement of $I$ as $I^c=[r]\setminus I$.
\end{definition}

For a nonempty $I\subseteq [r]$ with interval partition $\bigsqcup_{x=1}^{n(I)}[i_x,j_x]$, we define 
\begin{equation}\label{eq: defn of alpha_I}
    \alpha_I=\sum_{i\in I}\alpha_i=\sum_{x=1}^{n(I)}\alpha_{i_x,j_x}.\notag
\end{equation} 
Observe that every sum of distinct simple roots  can be written in the form $\sum_{i\in I}\alpha_i$ for some subset $I\subseteq [r]$, which, in turn, has a corresponding expression $\sum_{x=1}^{n(I)}\alpha_{i_x,j_x}$. Therefore to extend the work of Harris and Harry, we compute the Weyl alternation sets $\mathcal{A}(\hr,\alpha_I)$ for any nonempty $I\subseteq[r]$. 

An interesting detail of these aforementioned Weyl alternation sets is their enumeration. Recall that choosing nonconsecutive integers in $[n]$ is enumerated by $F_{n+2}$, where $F_i$ denotes the $i$-th Fibonacci number (see \Cref{prop: count is Fibonacci}).
\begin{definition}
    For any $n\geq 3$, define the \textbf{Fibonacci number}, denoted $F_n$, as the recursive formula:
    \begin{equation}
        F_n = F_{n-1}+F_{n-2}\notag
    \end{equation}
    with $F_1=F_2=1$.
\end{definition}
Consequently, $\mathcal{A}(\hr,0)$ is enumerated by a Fibonacci number \cite{HarrisThesis}, and $\mathcal{A}(\hr,\alpha_{i,j})$ is enumerated by a product of two Fibonacci numbers \cite{harry2024computingqmultiplicitypositiveroots}, dependent on the choice of positive root $\alpha_{i,j}$.
In \Cref{section: 2 Weyl alt sets}, we show that when $\lambda$ is the highest root and $\mu=\alpha_I$ for some nonempty $I\subseteq [r]$, the Weyl alternation sets are enumerated by products of Fibonacci numbers.
Additionally, we show that the products of Fibonacci numbers are determined exactly by the indexing set $I$.
\begin{restatable*}[]{corollary}{corr}\label{size of the alt set}
    Let $\hr$ be the highest root in $\speclin$ and let $I\subseteq [r]$ be nonempty with interval partition $I=\bigsqcup_{x=1}^{n(I)}[i_x,j_x]$. Then
    \begin{align}
        |\mathcal{A}(\hr,\alpha_I)|=\prod_{x=0}^{n(I)+1}F_{\ell_x},\notag
    \end{align}
    where $F_n$ denotes the $n$-th Fibonacci number, $\ell_0=i_1$, $\ell_{n(I)+1}=r-j_{n(I)}+1$, and for each $x\in[n(I)]$, $\ell_x=i_{x+1}-j_x+1$.
\end{restatable*}

In \Cref{section: 3 q-analogs of 101}, we prove a special case of $m_q(\hr,\alpha_I)$ when $\alpha_I=\alpha_{1,i}+\alpha_{i+j+1,r}$, provided that $i\in[r-2]$ and $j\in[r-i-1]$, then $m_q(\hr,\alpha_I)$ is a difference of powers of $q$. 
\begin{restatable*}[]{theorem}{difference}\label{thm: difference of q's}
    Let $\hr$ be the highest root in $\speclin$. 
    Fix  $i \in [r-2]$ and let $j\in[r-i-1]$.  
    Then
    \begin{equation}
        m_q(\hr,\alpha_{1,i}+\alpha_{i+j+1,r})=q^j-q^{j-1}.\notag
    \end{equation}
\end{restatable*}
Lastly, in \Cref{section: 4 main results}, by utilizing \Cref{kim result} and \Cref{thm: difference of q's} 
, we provide a method of factoring $m_q(\hr,\alpha_I)$ into a product of weight $q$-multiplicities to determine the weight $q$-multiplicity $m_q(\hr,\alpha_I)$ for any nonempty $I\subseteq [r]$.
\begin{restatable*}[]{corollary}{THEBIGONE}\label{thm: the big one}
    Let $\hr$ be the highest root in $\speclin$, and let $I\subseteq[r]$ be nonempty.
    Then
    \begin{equation}
        m_q(\hr,\alpha_I)=(q-1)^{n(I)-1}\cdot q^{r-|I|-n(I)+1}.\label{eq: KWMF final poly}
    \end{equation}
\end{restatable*}
\section{Weyl Alternation Sets and Products of Fibonacci numbers}\label{section: 2 Weyl alt sets}
In this section, we determine the Weyl alternation sets $\thealtset$ for any nonempty $I\subseteq [r]$. To begin, we recall a result from
\cite{anderson2025supportkostantsweightmultiplicity} (see Corollary 3.2)
which proves that when $\lambda$ is the highest root $\hr$ and $\mu$ is a positive root, if $\sigma\in \mathcal{A}(\hr,\mu)$ and $\sigma'$ is a reduced subword contained in a reduced word expression of $\sigma$, then $\sigma'\in\mathcal{A}(\hr,\mu)$. That is to say, to determine which elements of the Weyl group are contained in the Weyl alternation set $\mathcal{A}(\hr,\alpha_I)$, we first determine which simple reflections $s_i$ are contained in the alternation set, as any element of the alternation set will be a product of these simple reflections. Then we prove some properties pertaining to the structure of these Weyl alternation set elements which help in our characterization. 
Now we show that, for a given index set $I\subseteq[r]$, every simple reflection whose index is contained in $I\cup \{1,r\}$ does not exist in the alternation set.
\begin{proposition}\label{prop: no s1 sr}
    Let $\hr$ be the highest root of $\speclin$, and let $I\subseteq [r]$ be nonempty. 
    If $i\in I\cup\{1,r\}$, then $s_i\notin\mathcal{A}(\hr,\alpha_I)$. 
    Otherwise $s_i\in\thealtset$.
\end{proposition}
\begin{proof}
    Fix a nonempty $I\subseteq[r]$. To start, we show explicitly that neither $s_1$ nor $s_r$ are contained in the Weyl alternation set $\thealtset$. To this end, we compute
    \begin{align}
        s_1(\hr+\rho)-\rho-\alpha_I &= -\alpha_1+\alpha_2+\cdots+\alpha_r-\alpha_I\label{eq: apply s1}, \\
        \intertext{and}
        s_r(\hr+\rho)-\rho-\alpha_I &= \alpha_1+\alpha_2+\cdots+\alpha_{r-1}-\alpha_r-\alpha_I\label{eq: apply sr}.
    \end{align}
    In \Cref{eq: apply s1} 
    the coefficient of $\alpha_1$ is negative, and in  \Cref{eq: apply sr}, the coefficient of $\alpha_r$ is negative, hence $s_1,s_r\notin\mathcal{A}(\hr,\alpha_I)$.
    Consider $s_i$, where $i\in I\setminus \{1,r\}$. 
    We compute
    \begin{equation}
        s_i(\hr+\rho)-\rho-\alpha_I = \big(\alpha_1+\alpha_2+\ldots+\alpha_{i-1}+\alpha_i+\alpha_{i+1}+\cdots+\alpha_r\big)-\alpha_i-\alpha_I\label{eq: si acting on hr + rho}.
    \end{equation}
    Recall $I$ has interval partition $I=\bigsqcup_{x=1}^{n(I)}[i_x,j_x]$.
    If $i\in[i_x,j_x]$ for some $x$, then $\alpha_I=\sum_{x=1}^{n(I)}\alpha_{i_x,j_x}$ contains the term $\alpha_i$. More specifically, in \Cref{eq: si acting on hr + rho}, the $\alpha_i$ term has a coefficient of $-1$, and hence $s_i\notin\mathcal{A}(\hr,\alpha_I)$. 
    Now suppose that $i\notin I\cup\{1,r\}$, then each term in the expression $-\alpha_i-\alpha_I$ in \Cref{eq: si acting on hr + rho} all have distinct indices. 
    Therefore, the right hand side of \Cref{eq: si acting on hr + rho} is a nonnegative integral sum of positive roots, and hence $s_i\in \thealtset$ when $i\notin I$, which completes the proof.
\end{proof}
Next we recall a result that we utilize throughout. 
\begin{lemma}[Proposition 3.1.2 in \cite{HarrisThesis}]\label{lemma: spit out the alphas}
    Let $\hr$ be the highest root in $\speclin$. If $\sigma=s_{x_1}s_{x_2}\cdots s_{x_m}$, where $x_1,x_2,\ldots,x_m$ are nonconsecutive integers in the interval $[2,r-1]$, then 
    \[\sigma(\hr+\rho)=\hr+\rho - \sum_{i=1}^m \alpha_{x_i}.\]
\end{lemma}
Now we show that if $\sigma\in W$ contains any subword with indices forming a consecutive set of integers, then $\sigma\notin \mathcal{A}(\hr,\alpha_I)$ 
\begin{lemma}\label{lemma:consecutivegeneral}
    Let $\hr$ be the highest root in $\speclin$, and 
    let 
    $I\subseteq [r]$ be nonempty. 
    If $\sigma$ contains the word $s_{i+1}s_i$ or $s_is_{i+1}$ in it a reduced word expression, then $\sigma\notin \mathcal{A}(\hr,\alpha_I)$.
\end{lemma}
\begin{proof}
    By \Cref{prop: no s1 sr}, if $i\in I\cup\{1,r\}$, then $s_i \notin\thealtset$. Moreover,
    by \cite[Proposition 3.4]{Harris2016}, any $\sigma\in W$ containing $s_i$, with $i\in I\cup\{1,r\}$ in its reduced expression, will also not be in $\thealtset$. That is to say, to complete the proof, it suffices to show only the cases of when $\sigma=s_is_{i+1}$ or $\sigma=s_{i+1}s_i$.
    
    Consider the case where $i,i+1\in I^c\setminus\{1,r\}$ and observe that
    \begin{align}
        s_is_{i+1}(\hr+\rho)-\rho-\alpha_I &=s_i(\alpha_1+\alpha_2+\ldots+\alpha_i+\alpha_{i+2}+\ldots +\alpha_r+\rho) -\rho -\alpha_I \notag\\
        &=\alpha_1+\alpha_2+\ldots+\alpha_{i-1}-\alpha_i+\alpha_{i+1}+\ldots+\alpha_r-\alpha_I.\label{eq: i i+1 case}
        \intertext{Since $i\notin I\cup\{1,r\}$, then \Cref{eq: i i+1 case} has the term $-\alpha_i$ and thus $s_is_{i+1}\notin \thealtset$. Similarly,}
        s_{i+1}s_i(\hr+\rho)-\rho-\alpha_I &= \alpha_1+\alpha_2+\ldots + \alpha_{i-1} - \alpha_{i-1} + \alpha_{i+2} +\ldots + \alpha_r - \alpha_I.\label{eq: i+1 i case}
    \end{align}
    Notice that \Cref{eq: i+1 i case} contains a term with a negative coefficient, namely $-\alpha_{i+1}$
    , and so $s_{i+1}s_i\notin \thealtset$.
    This completes the proof.
\end{proof}
Fix a nonempty indexing set $I\subseteq[r]$. 
The number of intervals in the interval partition of $I$ fully determines the number of intervals in the interval partition of $I^c$, as we show next. 
\begin{lemma}\label{def:n(I^c)}
Fix a nonempty indexing set $I\subseteq[r]$.
Then the number of intervals in the interval partition of $I^c$ is given by 
\begin{align}
    n(I^c)=\begin{cases}
        n(I)-1 & \mbox{if $1,r\in I$} \\
        n(I) & \mbox{if $1\notin I$ and $r\in I$, or if $1\notin I$ and $r\in I$}\\
        n(I)+1 & \mbox{if $1,r\notin I$}.
    \end{cases}
\end{align}
\end{lemma}
\begin{proof}
    Let $I=\bigsqcup_{x=1}^{n(I)}[i_x,j_x]$, as defined in \Cref{def: interval partition}. To prove the claim, in each case we proceed by induction on $n(I)$. \\ 
    \textbf{Case 1:} Assume both $1,r\in I$. \\
    If $n(I)=1$, then $I=[r]$ and, therefore, $n(I^c)=0$ as desired. Now suppose $I$ contains $n(I)>1$ many intervals. That is to say $I$ is of the form
    \begin{equation}
        I = [1,j_1]\sqcup [i_2,j_2] \sqcup \cdots \sqcup [i_{n(I)-1},j_{n(I)-1}]\sqcup [i_{n(I)},r].\notag
    \end{equation}
    This implies that $I^c$ is of the form 
    \begin{equation}
        I^c = [j_1+1,i_2-1] \sqcup [j_2+1] \sqcup \cdots [j_{n(I)-2}+1,i_{n(I)-1}-1] \sqcup [j_{n(I)-1}+1,i_{n(I)}-1],\notag
    \end{equation}
    and so $n(I^c)=n(I)-1$ as claimed.\\
    \textbf{Case 2:} Assume $1\in I$ and $r\notin I$. \\
    If $n(I)=1$, then $I=[2,r]$ which implies $n(I^c)=n(I)=1$ as desired. Now suppose $I$ contains $n(I)>1$ many intervals. That is to say $I$ is of the form
    \begin{equation}
        I = [1,j_1] \sqcup [i_2,j_2] \sqcup \cdots \sqcup [i_{n(I)-1},j_{n(I)-1}] \sqcup [i_{n(I)},j_{n(I)}],\notag
    \end{equation}
    with $j_{n(I)}<r$. This implies that $I^c$ is of the form
    \begin{equation}
        I^c = [j_1+1,i_2-1]\sqcup [j_2+1,i_3-1]\sqcup \cdots \sqcup [j_{n(I)-1}+1,i_{n(I)-1}-1] \sqcup [j_{n(I)}+1,r],\notag
    \end{equation}
    and so $n(I^c)=n(I)$ as claimed.\\
    \textbf{Case 3:} Assume $1\notin I$ and $r\in I$.\\
    If $n(I)=1$, then $I=[1,r-1]$ which implies $n(I^c)=n(I)=1$ as desired. Now suppose $I$ contains $n(I)>1$ many intervals. That is to say $I$ is of the form
    \begin{equation}
        I = [i_1,j_1] \sqcup [i_2,j_2] \sqcup \cdots \sqcup [i_{n(I)-1},j_{n(I)-1}]\sqcup [i_{n(I)},r],\notag
    \end{equation}
    with $i_1>1$. This implies that $I^c$ is of the form
    \begin{equation}
        I^c = [1,i_1-1] \sqcup [j_1+1,i_2-1] \sqcup \cdots \sqcup [j_{n(I)-1}+1,i_{n(I)}-1] \sqcup [j_{n(I)-1}+1,i_{n(I)}-1],\notag
    \end{equation}
    and so $n(I^c)=n(I)$ as claimed. \\
    \textbf{Case 4:} Assume both $1,r\notin I$.\\
    If $n(I)=1$, then $I=[2,r-1]$ which implies $n(I^c)=n(I)+1=2$ as desired. Now suppose $I$ contains $n(I)>1$ many intervals. That is to say $I$ is of the form
    \begin{equation}
        I = [i_1,j_1]\sqcup [i_2,j_2] \sqcup \cdots \sqcup [i_{n(I)-1},j_{n(I)-1}]\sqcup [i_{n(I)},j_{n(I)}],\notag
    \end{equation}
    with $i_1>1$ and $j_{n(I)}<r$. This implies that $I^c$ is of the form
    \begin{equation}
        I^c = [1,i_1-1] \sqcup [j_1+1,i_2-1] \sqcup \cdots \sqcup [j_{n(I)-1}+1,i_{n(I)-1}-1], \sqcup [j_{n(I)}+1,r]\notag
    \end{equation}
    and so $n(I^c)=n(I)+1$ as claimed, which completes the proof. 
\end{proof}

Now we present a result that is the first step in providing a means of converting $m_q(\hr,\alpha_I)$ into sums 
or products of weight $q$-multiplicites in lower ranks. In \Cref{section: 3 q-analogs of 101}, we utilize this result to convert the weight $q$-multiplicity into a sum of $q$-multiplicities, and in \Cref{section: 4 main results}, we utilize the following result to convert $m_q(\hr,\alpha_I)$ into a product of weight $q$-multiplicities.
\begin{lemma}\label{lemma: alphas get spit out}
    Fix $I\subseteq[r]$ and let $\mathcal{I}$ be a collection of nonconsecutive integers in $I^c\setminus\{1,r\}$.
    If $\mathcal{I}$ is nonempty and $\sigma=\prod_{i\in \mathcal{I}}s_{i}$, then 
    \begin{equation}
    \sigma(\hr+\rho)=\hr+\rho-\sum_{i\in\mathcal{I}}\alpha_{i}.\notag
    \end{equation}
    If $\mathcal{I}$ is empty, then $\sigma=1$ and we set $\sum_{i\in\mathcal{I}}\alpha_i=0$.
\end{lemma}

\begin{proof}
    Let $I\subseteq[r]$ be nonempty, and let $I^c$ have interval partition $I^c=I_1\sqcup I_2\sqcup \cdots \sqcup I_{n(I^c)}$, as defined in \Cref{def: interval partition}.
    For each $x\in [n(I^c)]$, select a (possibly empty) collection of nonconsecutive indices from $I_x$ and denote these by $\mathcal{I}_x=\{x_1,x_2,\ldots,x_m\}$. 
    Then, for each $\mathcal{I}_x=\{x_1,x_2,\ldots,x_m\}$, we let $\sigma_x=s_{x_1}s_{x_2}\cdots s_{x_m}$, and we set $\sigma_x=1$ if $\mathcal{I}_x$ is empty.
    By \Cref{lemma: spit out the alphas}, if $\sigma=\sigma_1\sigma_2\cdots \sigma_{n(I^c)}$, then
    \begin{equation}
    \sigma(\hr+\rho)=\sigma_1\sigma_2\cdots \sigma_{n(I^c)}(\hr+\rho)=\hr+\rho - \sum_{i\in\mathcal{I}_1}\alpha_{i}- \sum_{i\in\mathcal{I}_2}\alpha_{i}-\cdots - \sum_{i\in\mathcal{I}_{n(I^c)}}\alpha_{i}=\hr+\rho-\sum_{j\in[n(I^c)]}\left(\sum_{i\in\mathcal{I}_j}\alpha_{i}\right).\notag
    \end{equation}
    To complete the proof, it suffices to let $\mathcal{I}=\displaystyle\bigcup_{j=1}^{n(I^c)}\mathcal{I}_j$, in which case $\displaystyle\sum_{j\in[n(I^c)]}\left(\displaystyle\sum_{i\in\mathcal{I}_j}\alpha_{i}\right)=\sum_{i\in\mathcal{I}}\alpha_i$.
\end{proof}
We are now ready to give a complete characterization of the elements in the Weyl alternation sets $\thealtset$ for any nonempty $I\subseteq [r]$. 
\begin{theorem}\label{thm:whats in the alt set} 
Let $\hr$ be the highest root in $\speclin$.
    Let $I\subseteq[r]$ be nonempty, and let $\mathcal{I}$ be a (possibly empty) collection of nonconsecutive integers in $I^c\setminus\{1,r\}$.     
    Then $\sigma\in\thealtset$ if and only if 
    \[\sigma=1 \quad \text{or} \quad \sigma=\prod_{i\in \mathcal{I}}s_i.\]
\end{theorem}
\begin{proof}
    $(\impliedby)$ We begin by showing that the identity element is contained in the alternation set.
    \begin{equation}
        1(\hr+\rho)-\rho-\alpha_I = \hr + \rho -\rho - \alpha_I = \hr -\alpha_I=\alpha_{I^c}.\label{eq: identity case}
    \end{equation} 
    \Cref{eq: identity case} is a sum of positive roots, and hence the identity is contained in the Weyl alternation set $\thealtset$.
    Now consider 
$\sigma=\prod_{i\in\mathcal{I}}s_i$. 
Recall that $s_i$ and $s_j$ commute whenever $i$ and $j$ are nonconsecutive integers. 
Then, by \Cref{lemma: alphas get spit out}, observe that
    \begin{align}
        \sigma(\hr+\rho)-\rho-\alpha_I=\left(\prod_{i\in\mathcal{I}}s_i\right)(\hr+\rho)-\rho-\alpha_I&= 
        \alpha_{I^c} -\sum_{i\in\mathcal{I}}\alpha_{i}.\label{eq: double sum}
    \end{align}
    Since $\mathcal{I}\subseteq I^c\setminus\{1,r\}$, every term contained in the sum in \Cref{eq: double sum} has a distinct index contained in $I^c\setminus\{1,r\}$. Thus \Cref{eq: double sum}
    reduces to a sum of positive roots. 
    Hence, $\sigma=\prod_{i\in\mathcal{I}}s_i\in\mathcal{A}(\hr,\alpha_I)$, whenever $\mathcal{I}$ is a set of nonconsecutive integers from the set $I^c\setminus\{1,r\}$.
    
 $(\implies)$
    Suppose $\sigma\in\mathcal{A}(\hr,\alpha_I)$. \Cref{eq: identity case} shows that the identity is contained in the alternation set. 
    By \Cref{prop: no s1 sr}, if $s_x$ appears in any reduced word expression for $\sigma$, then $x\notin I\cup\{1,r\}$. 
    By \Cref{lemma:consecutivegeneral}, if 
    $s_is_j$ appears in any reduced expression for $\sigma$, then $i$ and $j$ are nonconsecutive indices in the set $I^c\setminus\{1,r\}$. 
    Together, this implies that if $\sigma\in\thealtset$, then $\sigma=1$ or $\sigma=s_{\ell_1}s_{\ell_2}\cdots s_{\ell_m}$  where $\{\ell_1,\ell_2,\ldots,\ell_m\}$ is a set of nonconsecutive integers in $I^c\setminus\{1,r\}$, which completes the proof.   
\end{proof}
We now recall the following enumerative result, and refer the reader to \cite[Proposition 2.3]{harry2024computingqmultiplicitypositiveroots} for a proof. 
\begin{proposition}\label{prop: count is Fibonacci}
    The number of ways to select nonconsecutive numbers from $[n]=\{1,2,\ldots,n\}$ is given by $F_{n+2}$, where $F_i$ denotes the $i$-th Fibonacci number.
\end{proposition}
Now we prove that the Weyl alternation sets $\thealtset$ for any nonempty $I\subseteq[r]$ are enumerated by a product of Fibonacci numbers dependent on the indexing set $I$.
\corr
\begin{proof}
    By \Cref{thm:whats in the alt set}, $\sigma\in\thealtset$ if and only if 
    \[\sigma=1 \quad \text{or}\quad \sigma=\prod_{i\in\mathcal{I}}s_i,\]
    where $\mathcal{I}$ is a set of nonconsecutive integers in $I^c\setminus\{1,r\}$. Recall that $I^c$ has interval partition 
    \[I^c = I_1 \sqcup I_2 \sqcup \cdots \sqcup I_{n(I^c)},\]
    where $n(I^c)$ is as defined in \Cref{def: interval partition}.  
    For each $x\in [n(I^c)]$, select $\mathcal{I}_x=\{x_1,x_2,\ldots,x_m\}$, which is 
    a (possibly empty) collection of nonconsecutive integers from $I_x$.
    Then, for each $\mathcal{I}_x=\{x_1,x_2,\ldots,x_m\}$, we let $\sigma_x=s_{x_1}s_{x_2}\cdots s_{x_m}$, and we set $\sigma_x=1$ if $\mathcal{I}_x$ is empty. That is to say, if $\sigma\in\thealtset$, then either
    \[\sigma=1 \quad \text{or}\quad \sigma=\sigma_1\sigma_2\cdots\sigma_{n(I^c)}.\]
    Moreover, $\mathcal{I}_x\cap \mathcal{I}_y=\emptyset$ for all distinct $x,y\in[n(I^c)]$.
    Thus, we count the number of elements in the Weyl alternation set $\thealtset$ by counting all of the possible $\sigma$ constructed as products of simple reflections whose indices come from the sets $\mathcal{I}_x$ of nonconsecutive integers. 
    That is, we can take the product of the counts by building each $\sigma_x$ individually. 
    
    Observe that if $I=\bigsqcup_{x=1}^{n(I)}[i_x,j_x]$ and $1\notin I$, then by \Cref{thm:whats in the alt set} and \Cref{prop: no s1 sr}, 
    \begin{equation}
        \sigma_1=\prod_{i\in \mathcal{I}_1}s_i,\notag
    \end{equation} 
    where
    $\mathcal{I}_1$ is a (possibly empty) set of nonconsecutive integers in the interval $[2,i_1-1]$. By shifting the indices by $-1$, this is exactly equivalent to choosing nonconsecutive integers in the interval $[i_1-2]$ and so by \Cref{prop: count is Fibonacci}, there are $F_{i_1}$ many such subsets.  
    Notice that if $1\in I$, then $i_1=1$ and $F_{i_1}=F_1=1$, and so in both cases the count is determined by $F_{i_1}$. 
    If $r\notin I$, then by \Cref{thm:whats in the alt set} and \Cref{prop: no s1 sr},
    \[\sigma_{n(I)}=\prod_{i\in\mathcal{I}_{n(I)}}s_i,\]
    such that $\mathcal{I}_{n(I)}$ is a set of nonconsecutive integers in the interval $[j_{n(I)}+1,r-1]$. By shifting the indices by $-j_{n(I)}$, this is exactly equivalent to choosing nonconsecutive integers in the interval $[r-j_{n(I)}-1]$ and so by \Cref{prop: count is Fibonacci}, there are $F_{r-j_{n(I)}+1}$ many such subsets. 
    Notice if $r\in I$, then $j_{n(I)}=r$ and $F_{r-j_{n(I)}+1}=F_{1}=1$, and so in both cases the count is determined by $F_{r-j_{n(I)}+1}$.

    Now consider $\sigma_x$ for each $x\in[n(I)]$. Either $\sigma_x=1$, or $\sigma_x=\prod_{i\in\mathcal{I}_x}s_i$ where 
    $\mathcal{I}_x$ is a set of nonconsecutive integers in the interval $I_x = [j_x+1,i_{x+1}-1]$. Notice that by shifting the indices by $-j_x$, this is exactly equivalent to choosing nonconsecutive integers in the interval $[i_{x+1}-j_x-1]$. So, by \Cref{prop: count is Fibonacci}, there are $F_{i_{x+1}-j_x+1}$ many such subsets. Therefore, the cardinality of $\thealtset$ is enumerated by the product  
    \begin{equation}
        \left|\mathcal{A}(\hr,\alpha_I)\right|=\prod_{x=0}^{n(I)+1}F_{\ell_x},\notag
    \end{equation}
    where $F_n$ denotes the $n$-th Fibonacci number, $\ell_0=i_1$, $\ell_{n(I)+1}=r-j_{n(I)}+1$, and for each $x\in[n(I)]$, $\ell_x=i_{x+1}-j_x+1$, which completes the proof.
\end{proof}

\section{A $q$-analog for a sum of two positive roots }\label{section: 3 q-analogs of 101}
In this section, we let $I\subseteq[r]$ have interval partition $I=[1,i]\sqcup[i+j+1,r]$, with $i\in[r-2]$ and $j\in [r-i-1]$.   
We show that in this case, the weight $q$-multiplicity $m_q(\hr,\alpha_I)$ 
is a difference of powers of $q$. 
This result, along with \Cref{kim result}, are utilized in \Cref{section: 4 main results} to prove \Cref{thm: the big one}, where we consider any  nonempty subset $I\subseteq[r]$. 
We begin by recalling a property of Kostant's partition function that greatly reduces the computational complexity of determining $m_q(\hr,\alpha_I)$. 
\begin{lemma}[Lemma 3.4 in \cite{harry2024computingqmultiplicitypositiveroots}]
    \label{lemma: rescaling q partition is cool}
    Let $i,j$ be integers such that $1\leq i \leq j\leq r$. Then,
    \[\wp_q(\alpha_{i,j})=\qpartition{\alpha_{1,j-i+1}}=q(q+1)^{j-i}.\]
\end{lemma}
Consider evaluating the partition function on two positive roots whose sets of indices are disjoint. 
For example, let $r=7$ and consider 
\begin{equation}
    \wp_q\Big((\alpha_2+\alpha_3)+(\alpha_5+\alpha_6)\Big).\notag
\end{equation}
Computing the number of ways to write  $(\alpha_2+\alpha_3)+(\alpha_5+\alpha_6)$ as a sum of positive roots is identical to computing the number of ways to write $\alpha_2+\alpha_3$ and $\alpha_5+\alpha_6$ individually. Since any choice of positive root containing terms with indices $2$ and $6$ must also include every integer betwen $2$ and $6$, it is equivalent to taking the product of the partition function on the positive roots independently. More precisely,
\begin{equation}
    \wp_q\Big((\alpha_2+\alpha_3)+(\alpha_5+\alpha_6)\Big)=\wp_q(\alpha_2+\alpha_3)\cdot \wp_q(\alpha_5+\alpha_6).\notag
\end{equation}
This observation works whenever we want to compute the value of the Kostant's partition function on a sum of positive roots whose indices on the simple roots of each positive root are disjoint. We provide a precise result for this observation now.
\begin{lemma}\label{lemma: split up partition function}
    Let $I\subseteq [r]$ be any nonempty indexing set with interval partition $I=\bigsqcup_{x=1}^{n(I)}[i_x,j_x]$. Then,
    \begin{equation}
        \wp_q(\alpha_I) = \prod_{x=1}^{n(I)}\wp_q(\alpha_{i_x,j_x}).\notag
    \end{equation}
\end{lemma}
\begin{proof}
    
    First, recall that $\alpha_I$ is a sum of $\alpha_{i,j}$, where each $\alpha_{i,j}$ is contained in a distinct and disjoint interval from every other $\alpha_{i',j'}$ in the sum. This means that, when computing the number of ways one can express $\alpha_I$ as a sum of positive roots, the number of components for $\alpha_{i,j}$ is independent of every other positive root. Simply put,
    \[\wp_q(\alpha_I)=\prod_{k=1}^n\wp_q(\alpha_{i_k,j_k})\]
    which completes the proof.
\end{proof}
We are now ready to prove that the weight $q$-multiplicity $m_q(\hr,\alpha_I)$ for $\alpha_I=\alpha_{1,i}+\alpha_{i+j+1}$ is a difference of powers of $q$. 

\difference
\begin{proof}
    Throughout the proof, we often consider ranks less than rank $r$. To avoid confusion, we decorate
    $\hr$, $\rho$, and $\mathcal{A}$ with a subscript 
    to specify rank. 
    By \Cref{thm:whats in the alt set}, $\sigma\in\mathcal{A}_r(\hr_r,\alpha_{1,i}+\alpha_{i+j+1,r})$ if and only if 
    \begin{equation}
        \sigma = 1 \quad \text{ or } \quad \sigma=\sigma_1,\notag
    \end{equation}
    with  
    $\sigma_1=\prod_{x\in\mathcal{I}_1}s_x$, 
    where $\mathcal{I}_1$  
    is a set of nonconsecutive integers in the interval $[i+1,i+j]$. Observe that when $\mathcal{I}_1$ is empty, then $\sigma_1=1$.

    We now proceed by considering what elements are in the Weyl alternation set for a fixed $j$ and use this to compute the $q$-multiplicity.
    If $j=1$, then the Weyl alternation set is precisely
    \begin{equation}
        \mathcal{A}_r(\hr_r,\alpha_{1,i}+\alpha_{i+2,r}) = \{1,~s_{i+1}\}.\notag
    \end{equation}
    Since $(-1)^{\ell(s_{i+1})}=-1$, the weight $q$-multiplicity reduces to
    \begin{align*}
        m_q(\hr_r,\alpha_{1,i}+\alpha_{i+2,r}) &= \wp_q(1(\hr_r+\rho_r)-\rho_r-\alpha_{1,i}-\alpha_{i+2,r})-\wp_q(s_{i+1}(\hr_r+\rho_r)-\rho_r-\alpha_{1,i}-\alpha_{i+2,r}) \\
        &= \wp_q(\alpha_{i+1}) - \wp_q(0) \\ 
        &= q-1\\
        &=q^1-q^{1-1},
    \end{align*}
    which is in the desired form.
    If $j=2$, then the Weyl alternation set is precisely
    \begin{equation}
        \mathcal{A}_r(\hr_r,\alpha_{1,i}+\alpha_{i+3,r}) = \{1,s_{i+1},s_{i+2}\}.\notag
    \end{equation}
    Then the weight $q$-multiplicity is
    \begin{align*}
        m_q(\hr_r,\alpha_{1,i}+\alpha_{i+3}) &= \wp_q(1(\hr_r+\rho_r)-\rho_r-\alpha_{1,i}-\alpha_{i+3,r})-\wp_q(s_{i+1}(\hr_r+\rho_r)-\rho_r-\alpha_{1,i}-\alpha_{i+3,r}) \\
        &\quad\quad -\wp_q(s_{i+2}(\hr_r+\rho_r)-\rho_r-\alpha_{1,i}-\alpha_{i+3,r}) \\
        &= \wp_q(\alpha_{i+1,i+2}) - \wp_q(\alpha_{i+2}) - \wp_q(\alpha_{i+1}) \\
        &= q(q+1)^{1}-q-q \\
        &= q^2-q,
    \end{align*}
    which is in the desired form. 
    
    Now consider $j\geq3$.
    In this case, there exist elements in the Weyl alternation set that have lengths greater than one. 
    To address this, we partition the Weyl alternation set into four disjoint subsets depending on whether or not the Weyl group elements contain $s_{i+1}$, or $s_{i+j}$ in a reduced word expression.
    To this end, let $\sigma_1'$ denote an element of the Weyl alternation set  $\mathcal{A}_r(\hr_r,\alpha_{1,i}+\alpha_{i+j+1,r})$ that does not contain either $s_{i+1}$ or $s_{i+j}$ in a reduced word expression.
    Thus, every $\sigma\in \mathcal{A}_r(\hr_r,\alpha_{1,i}+\alpha_{i+j+1,r})$ can be expressed as
    \begin{equation}
        \sigma = \begin{cases}
        \sigma_1' & \text{$\sigma$ does not contain either $s_{i+1}$ or $s_{i+j}$ in a reduced word expression} \\
        s_{i+1}\sigma_1' & \text{$\sigma$ contains $s_{i+1}$ and does not contain $s_{i+j}$} \\
        \sigma_1's_{i+j} & \text{$\sigma$ contains $s_{i+j}$ and does not contain $s_{i+1}$} \\
        s_{i+1}\sigma_1's_{i+j} & \text{$\sigma$ contains both $s_{i+1}$ and $s_{i+j}$ in a reduced word expression}.\notag
    \end{cases}
    \end{equation}
    We define the sets $A,B,C,D$ as follows:
    \begin{itemize}
        \item $A=\{\sigma\in\mathcal{A}_r(\hr_r,\alpha_{1,i}+\alpha_{i+j+1,r}): \sigma=\sigma_1'\}$, 
        \item $B=\{\sigma\in\mathcal{A}_r(\hr_r,\alpha_{1,i}+\alpha_{i+j+1,r}): \sigma=s_{i+1}\sigma_1'\}$,
        \item $C=\{\sigma\in\mathcal{A}_r(\hr_r,\alpha_{1,i}+\alpha_{i+j+1,r}): \sigma=\sigma_1's_{i+j}\}$, and 
        \item $D=\{\sigma\in\mathcal{A}_r(\hr_r,\alpha_{1,i}+\alpha_{i+j+1,r}): \sigma=s_{i+1}\sigma_1's_{i+j}\}$. 
    \end{itemize}
    This partitions the Weyl alternation set into the disjoint subsets $A$, $B$, $C$, and $D$. Thus, the weight $q$-multiplicity satisfies
    \begin{align}
    m_q(\hr_r,\alpha_{1,i}+\alpha_{i+j+1,r}) 
    &= \sum_{\sigma\in A} (-1)^{\ell(\sigma)}\qpartition{\sigma(\hr_r+\rho_r)-\rho_r-\alpha_{1,i}-\alpha_{i+j+1,r}} \notag \\
    &\qquad + \sum_{\sigma\in B} (-1)^{\ell(\sigma)}\qpartition{\sigma(\hr_r+\rho_r)-\rho_r-\alpha_{1,i}-\alpha_{i+j+1,r}} \notag \\
    &\qquad +\sum_{\sigma\in C}  (-1)^{\ell(\sigma)}\qpartition{\sigma(\hr_r+\rho_r)-\rho_r-\alpha_{1,i}-\alpha_{i+j+1,r}} \notag \\
    &\qquad + \sum_{\sigma\in D} (-1)^{\ell(\sigma)}\qpartition{\sigma(\hr_r+\rho_r)-\rho_r-\alpha_{1,i}-\alpha_{i+j+1,r}}.
        \label{eq: difference of q broken down}
    \end{align}
    Note that the identity element in the Weyl alternation set $\mathcal{A}(\hr_r,\alpha_{1,i}+\alpha_{i+j+1,r})$ appears as a term in the sum corresponding to the set $A$
    since, by definition, $\sigma_1'$ can be the identity element, as the only requirement is that it does not contain $s_{i+1}$ nor $s_{i+j}$ in any reduced word expression for $\sigma_1'$. 

    To compute the weight $q$-multiplicity $m_q(\hr_r,\alpha_{1,i}+\alpha_{i+j+1,r})$ for $j\geq 3$, we prove that each sum in \Cref{eq: difference of q broken down}, through some manipulation and reindexing, is equivalent to computing the weight $q$-multiplicity $m_q(\hr_t,0)$, with $t\in\{j,j-1,j-2\}$ as described in \Cref{eq: pamela result}.  We proceed by considering each sum individually. \\
    Consider the sum $\sum_{\sigma\in A} (-1)^{\ell(\sigma)}\qpartition{\sigma(\hr_r+\rho_r)-\rho_r-\alpha_{1,i}-\alpha_{i+j+1,r}}$ in \Cref{eq: difference of q broken down}. For any arbitrary $\sigma=\sigma_1'$, the input of Kostant's partition function can be reduced to a partition function on a positive root minus a sum of simple roots. By \Cref{lemma: alphas get spit out}, for any $\sigma_1'\in A$ with $ \sigma_1'=\prod_{x\in\mathcal{I}_1}s_x$ where $\mathcal{I}_1$ is a set of nonconsecutive integers in the interval $[i+2,i+j-1]$, we have that
    \begin{align}
        \wp_q\left(\sigma_1'(\hr_r+\rho_r)-\rho_r-\alpha_{1,i}-\alpha_{i+j+1,r}\right) &=
        \wp_q\left(\hr_r - \alpha_{1,i} -\alpha_{i+j+1,r} - \sum_{x\in \mathcal{I}_1} \alpha_{x}\right)\notag \\
        &= \wp_q\left(\alpha_{i+1,i+j} - \sum_{x\in \mathcal{I}_1} \alpha_{x}\right).\notag
    \end{align}
    Note that $\sum_{x\in\mathcal{I}_1}\alpha_x=0$ if  $\sigma_1'=1$.  
    Therefore, the sum corresponding to the set $A$ can be expressed as
    \begin{align}
        \sum_{\sigma\in A} (-1)^{\ell(\sigma)}\qpartition{\sigma(\hr_r+\rho_r)-\rho_r-\alpha_{1,i}-\alpha_{i+j+1,r}} = \sum_{\sigma\in A}(-1)^{\ell(\sigma)}\qpartition{\alpha_{i+1,i+j}-\sum_{x\in\mathcal{I}_1}\alpha_x}\label{turn this into pamela result}.
    \end{align}
    Observe that every index on the simple roots appearing in the input of Kostant's partition function in \Cref{turn this into pamela result} have the form $i+c$ for some $c\in [j]$. 
    We shift the indices by $-i$ to set the smallest index to~$1$. Therefore, by \Cref{lemma: rescaling q partition is cool}, \Cref{turn this into pamela result} can be expressed as
    \begin{equation}
        \sum_{\sigma\in A} (-1)^{\ell(\sigma)}\qpartition{\sigma(\hr_r+\rho_r)-\rho_r-\alpha_{1,i}-\alpha_{i+j+1,r}} = \sum_{\sigma\in A'} (-1)^{\ell(\sigma)}\qpartition{\alpha_{1,j}-\sum_{x\in\mathcal{I}'_1}\alpha_x},\label{match this with EQ1}
    \end{equation}
    where $A'$ and $\mathcal{I}_1'$ denote the sets $A$ and $\mathcal{I}_1$ respectfully after shifting the indices. Moreover, notice that after the shift, $\sigma\in A'$ implies $\sigma=\prod_{x\in\mathcal{I}_1'}s_x$, where $\mathcal{I}_1'$ is a set of nonconsecutive integers in the interval $[2,j-1]$.
    
    Now consider $m_q(\hr_j,0)$ in rank $j$. By \Cref{pamela alternation result}, $\sigma\in \mathcal{A}_j(\hr_j,0)$ if and only if $\sigma=\prod_{x\in\mathcal{I}}s_x$, where $\mathcal{I}$ is a set of nonconsecutive integers in the interval $[2,j-1]$. Therefore, by \Cref{lemma: alphas get spit out}, we can express $m_q(\hr_j,0)$~as 
    \begin{align}
        m_q(\hr_j,0) &= \sum_{\sigma\in W} (-1)^{\ell(\sigma)}\qpartition{\sigma(\hr_j+\rho_j)-\rho_j} \notag\\
        &= \sum_{\sigma\in\mathcal{A}_j(\hr_j,0)} (-1)^{\ell(\sigma)}\qpartition{\hr_{j}-\sum_{x\in\mathcal{I}}\alpha_x}. \label{thing to match to}
    \end{align}
    Now we show that \Cref{match this with EQ1} is equivalent to \Cref{thing to match to}, despite these expressions being in different ranks. 
    To prove this, we first show that the sum contains an equivalent number of terms to the size of the alternation set $\mathcal{A}_j(\hr_j,0)$,
    and that every term in \Cref{match this with EQ1} has a unique corresponding term in \Cref{thing to match to}, which proves the claim.  
    First, recall that by \Cref{pamela alternation result} we know 
    \[|\mathcal{A}_j(\hr_j,0)|=F_j,\]
    and by \Cref{thm:whats in the alt set}, if $\sigma\in A'$, then $\sigma=1$ or $\sigma=\prod_{x\in\mathcal{I}'_1}s_x$, where $\mathcal{I}'_1$ is a set of nonconsecutive integers in the interval $[2,j-1]$. Notice that this is exactly equivalent to choosing nonconsecutive integers in the interval $[j-2]$, which is counted by $F_j$. This implies $|A'|=F_j$, and so indeed both sums have an equivalent number of terms. 
    Now we show that every term in \Cref{match this with EQ1} corresponds uniquely to a term in \Cref{thing to match to}. 
    First, consider the term indexed by the identity element in \Cref{match this with EQ1}; the corresponding term is of the form $(-1)^{\ell(1)}\qpartition{\alpha_{i+1,i+j}-0}$.  
    By \Cref{lemma: rescaling q partition is cool}, 
    \begin{align}
        (-1)^{\ell(1)}\qpartition{\alpha_{i+1,i+j}-0} &=
        \qpartition{\alpha_{i+1,i+j}} = q(q+1)^{j-1}. \notag
    \end{align}
    Likewise, the term indexed by the identity element in $m_q(\hr_j,0)$  is
    \begin{equation}
        (-1)^{\ell(1)}\qpartition{1(\hr_j+\rho_j)-\rho_j} = \qpartition{\alpha_{1,j}}= q(q+1)^{j-1}.\notag
    \end{equation}
    Therefore, the terms indexed by the identity elements are equal.
    Now consider any arbitrary non-identity element $\sigma\in A'$ 
    with corresponding term in \Cref{match this with EQ1} given by
    \begin{align}
        (-1)^{\ell(\sigma)}\qpartition{\alpha_{1,j}-\sum_{x\in\mathcal{I}'_1}\alpha_x},\label{phase 1 turn into rank j result}
    \end{align}
    with $\mathcal{I}'_1$ a set of nonconsecutive integers in the interval $[2,j-1]$.  
    By \Cref{pamela alternation result}, this is equivalent to a unique  $\sigma^*\in\mathcal{A}_j(\hr_j,0)$ using the same indices in $\mathcal{I}_1'$ for the simple reflections as in $\sigma$, but simple reflections in the Weyl group in rank $j$. This implies
    \begin{align}
        (-1)^{\ell(\sigma^*)}\qpartition{\hr_j-\sum_{x\in\mathcal{I}}\alpha_x} &= (-1)^{\ell(\sigma)}\qpartition{\alpha_{1,j}-\sum_{x\in\mathcal{I}'_1}\alpha_x}.\notag
    \end{align}
    Since our choice of non-identity element $\sigma\in A$ was arbitrary, every term in \Cref{turn this into pamela result} 
    corresponds to a unique term in $m_q(\hr_j,0)$, and the two equations are equivalent. 
    Therefore, by \Cref{eq: pamela result},
    \begin{equation}
        \sum_{\sigma\in A}(-1)^{\ell(\sigma)}\qpartition{\sigma(\hr_r+\rho_r)-\rho_r-\alpha_{1,i}-\alpha_{i+j+1,r}} = m_q(\hr_j,0) = \sum_{t=1}^j q^t. \label{DIFF OF EQS SUM 1}
    \end{equation} 
    Consider the sum $\sum_{\sigma\in B}(-1)^{\ell(\sigma)}\qpartition{\sigma(\hr_r+\rho_r)-\rho_r-\alpha_{i,1}-\alpha_{i+j+1,r}}$ in \Cref{turn this into pamela result}.
    For any arbitrary $\sigma=s_{i+1}\sigma_1'$, the input of Kostant's partition function can be reduced to a partition function on a positive root minus a sum of simple roots. By \Cref{lemma: alphas get spit out}, for any $s_{i+1}\sigma_1'\in B$ where $s_{i+1}\sigma_1'=s_{i+1}\cdot \prod_{x\in\mathcal{I}_1}s_x$ with $\mathcal{I}_1$ a set of nonconsecutive integers in the interval $[i+3,i+j-1]$, we have that
    \begin{align}
        &(-1)^{\ell(\sigma)}\qpartition{\sigma(\hr_r+\rho_r)-\rho_r-\alpha_{1,i}-\alpha_{i+j+1,r}} \notag\\
        &\hspace{5cm}= (-1)^{\ell(s_{i+1}\sigma_1')}\qpartition{\hr_r-\alpha_{1,i}-\alpha_{i+1}-\alpha_{i+j+1,r}-\sum_{x\in\mathcal{I}_1}\alpha_x}.\notag \\
        &\hspace{5cm}= (-1)^{\ell(s_{i+1}\sigma_1')}\qpartition{\alpha_{i+2,i+j}-\sum_{x\in\mathcal{I}_1}\alpha_x}.\notag
    \end{align}
    Note that $\sum_{x\in\mathcal{I}_1}\alpha_x=0$ if $\sigma'_1=1$. Since $\ell(s_{i+1}\sigma_1')=\ell(s_{i+1})+\ell(\sigma_1')$, the sum corresponding to the set $B$ can be expressed as
    \begin{align}
        \sum_{\sigma\in B}(-1)^{\ell(\sigma)}\qpartition{\sigma(\hr_r+\rho_r)-\rho_r-\alpha_{1,i}-\alpha_{i+j+1,r}} &= -\sum_{\substack{\sigma\in B \\ \sigma=s_{i+1}\sigma_1'}}(-1)^{\ell(\sigma_1')}\qpartition{\alpha_{i+2,i+j}-\sum_{x\in\mathcal{I}_1}\alpha_x}.\label{eq: case 2 in diff of q's}
    \end{align}
    Observe that every index on the simple roots appearing in the input of Kostant's partition function in \Cref{eq: case 2 in diff of q's} have the form $i+c$ for some $c\in[2,j]$.
    We shift the indices by $-i-1$ to set the smallest index to 1. 
    Therefore, by \Cref{lemma: rescaling q partition is cool}, \Cref{eq: case 2 in diff of q's} can be expressed as 
    \begin{align}
        \sum_{\substack{\sigma\in B \\ \sigma=s_{i+1}\sigma_1'}} -(-1)^{\ell(\sigma_1')}\qpartition{\alpha_{i+2,i+j}-\sum_{x\in\mathcal{I}_1}\alpha_x}=-\sum_{\sigma\in B'} (-1)^{\ell(\sigma)}\qpartition{\alpha_{1,j-1}-\sum_{x\in\mathcal{I}_1'}\alpha_x},\label{XXXX} 
    \end{align}
    where $B'$ and $\mathcal{I}_1'$ denote the sets $B$ and $\mathcal{I}_1$ respectfully after the shift in indices. Moreover, notice that after the shift, $\sigma\in B'$ implies $\sigma=\prod_{x\in\mathcal{I}_1'}s_x$, where $\mathcal{I}_1'$ is a set of nonconsecutive integers in the interval $[2,j-2]$.
    
    Now consider $m_q(\hr_{j-1},0)$ in rank $j-1$. By \Cref{pamela alternation result}, $\sigma\in \mathcal{A}_{j-1}(\hr_{j-1,0})$ if and only if $\sigma=\prod_{x\in\mathcal{I}}s_x$, where $\mathcal{I}$ is a set of nonconsecutive integers in the interval $[2,j-2]$. Therefore, by \Cref{lemma: alphas get spit out}, we can express $m_q(\hr_{j-1},0)$ as
    \begin{align}
        m_q(\hr_{j-1},0) &= \sum_{\sigma\in W}(-1)^{\ell(\sigma)}\qpartition{\sigma(\hr_{j-1}+\rho_{j-1})-\rho_{j-1}} \notag \\
        &= \sum_{\sigma\in\mathcal{A}_{j-1}(\hr_{j-1},0)}(-1)^{\ell(\sigma)}\qpartition{\hr_{j-1}-\sum_{x\in\mathcal{I}}\alpha_x}.\label{extentsion: case 2 pamela match}
    \end{align}
    Now we show that \Cref{XXXX} is equivalent to \Cref{extentsion: case 2 pamela match}, despite these expressions being in different ranks. To prove this, we first show that the sum contains an equivalent number of terms to the size of the alternation set $\mathcal{A}_{j-1}(\hr_{j-1},0)$, and that every term in \Cref{XXXX} has a unique corresponding term in \Cref{extentsion: case 2 pamela match}, which proves the claim. First, recall that by \Cref{pamela alternation result},
    \begin{equation}
        |\mathcal{A}_{j-1}(\hr_{j-1},0)|=F_{j-1},\notag
    \end{equation}
    and by \Cref{thm:whats in the alt set}, if $\sigma\in B'$ then $\sigma=1$ or $\sigma=\prod_{x\in\mathcal{I}_1'}s_x$, where $\mathcal{I}_1'$ is a set of nonconsecutive integers in the interval $[2,j-2]$. Notice that this is exactly equivalent to choosing nonconsecutive integers in the interval $[j-3]$, which is counted by $F_{j-1}$. This implies $|B'|=F_{j-1}$, and hence both sums have an equivalent number of terms. Now we show that every term in \Cref{XXXX} corresponds uniquely to a term in \Cref{extentsion: case 2 pamela match}. First, consider the term indexed by the identity element in \Cref{XXXX}, the corresponding term is of the form $(-1)^{\ell(1)}\qpartition{\alpha_{1,j-1}}$. By \Cref{lemma: rescaling q partition is cool}, 
    \begin{equation}
       (-1)^{\ell(1)}\qpartition{\alpha_{1,j-1}-0}=\qpartition{\alpha_{1,j-1}} = q(q+1)^{j-2}. \notag
    \end{equation}
    Likewise, the term indexed by the identity element in $m_q(\hr_{j-1},0)$ is
    \begin{equation}
        (-1)^{\ell(1)}\qpartition{1(\hr_{j-1}+\rho_{j-1})-\rho_{j-1}}= \qpartition{\alpha_{1,j-1}}=q(q+1)^{j-2}.\notag
    \end{equation}
    Therefore, the terms indexed by the identity elements are equal. Now consider any arbitrary non-identity element $\sigma\in B'$ with corresponding term in \Cref{XXXX} given by
    \begin{equation*}
        (-1)^{\ell(\sigma)}\qpartition{\alpha_{1,j-1}-\sum_{x\in\mathcal{I}_1'}\alpha_x},
    \end{equation*}
    with $\mathcal{I}_1'$ a set of nonconsecutive integers in the interval $[2,j-2]$. By \Cref{pamela alternation result}, this is equivalent to a unique $\sigma^*\in\mathcal{A}_{j-1}(\hr_{j-1},0)$, using the same indices in $\mathcal{I}_1'$ for the simple reflections as in $\sigma$, but simple reflections in the Weyl group in rank $j-1$. This implies
    \begin{equation}
        (-1)^{\ell(\sigma)}\qpartition{\alpha_{1,j-1}-\sum_{x\in\mathcal{I}'_1}\alpha_x},\notag
    \end{equation}
    with $\mathcal{I}'_1$  a set of nonconsecutive integers in the interval $[2,j-2]$. By \Cref{pamela alternation result}, this is equivalent to a unique $\sigma^*\in\mathcal{A}_{j-1}(\hr_{j-1},0)$, using the same indices in $\mathcal{I}_1'$ for the simple reflections as in $\sigma$, but simple reflections in the Weyl group in rank $j-1$. This implies 
    \begin{equation}
        (-1)^{\ell(\sigma^*)}\qpartition{\hr_{j-1}-\sum_{x\in\mathcal{I}}\alpha_x}=(-1)^{\ell(\sigma)}\qpartition{\alpha_{1,j-1}-\sum_{x\in\mathcal{I}_1'}\alpha_x}.\notag
    \end{equation}
    Since our choice of non-identity element $\sigma\in B'$ was arbitrary, every term in \Cref{XXXX} corresponds to a unique term in $m_q(\hr_{j-1},0)$, and so these two equations are equivalent. Therefore, by \Cref{eq: pamela result}, 
    \begin{equation}
        \sum_{\sigma\in B} (-1)^{\ell(\sigma)}\qpartition{\sigma(\hr_r+\rho_r)-\rho_r-\alpha_{1,i}-\alpha_{i+j+1,r}}= - m_q(\hr_{j-1},0) = - \sum_{t=1}^{j-1}q^t.\label{DIFF OF EQS SUM 2}
    \end{equation}
    
    Consider the sum $\sum_{\sigma\in C}(-1)^{\ell(\sigma)}\qpartition{\sigma(\hr_r+\rho_r)-\rho_r-\alpha_{1,i}-\alpha_{i+j+1,r}}$ in \Cref{turn this into pamela result}. 
    For any arbitrary $\sigma=s_{i+j}\sigma_1'$, the input of Konstant's partition function can be reduced to a partition function on a positive root minus a sum of simple roots. By \Cref{lemma: alphas get spit out}, for any $s_{i+j}\sigma_1'\in C$ where $s_{i+j}\sigma_1'=s_{i+j}\prod_{x\in\mathcal{I}_1}s_x$ with $\mathcal{I}_1$ a set of nonconsecutive integers in the interval $[i+2,i+j-2]$, we have that
    \begin{align}
        &(-1)^{\ell(\sigma)}\qpartition{\sigma(\hr_r+\rho_r)-\rho_r-\alpha_{1,i}-\alpha_{i+j+1,r}} \notag\\
        &\hspace{5cm}=(-1)^{\ell(s_{i+j}\sigma_1')}\qpartition{\hr_r-\alpha_{1,i}-\alpha_{i+j}-\alpha_{i+j+1,r}-\sum_{x\in\mathcal{I}_1}\alpha_x} \notag \\
        &\hspace{5cm}= (-1)^{\ell(s_{i+j}\sigma_1')}\qpartition{\alpha_{i+1,i+j-1}-\sum_{x\in\mathcal{I}_1}\alpha_x}. \notag
    \end{align}
    Note that $\sum_{x\in\mathcal{I}_1}\alpha_x=0$ if $\sigma'_1=1$.
    Since $\ell(s_{i+j}\sigma_1')=\ell(s_{i+j})+\ell(\sigma_1')=1+\ell(\sigma_1')$, 
    the sum corresponding to the set $C$ can be expressed as
    \begin{equation}
    \sum_{\sigma\in C}(-1)^{\ell(\sigma)}\qpartition{\sigma(\hr_r+\rho_r)-\rho_r-\alpha_{1,i}-\alpha_{i+j+1,r}}=-\sum_{\substack{\sigma\in C \\ \sigma=s_{i+j}\sigma_1'}} (-1)^{\ell(\sigma_1')}\qpartition{\alpha_{i+1,i+j-1}-\sum_{x\in\mathcal{I}'_1}\alpha_x}.\label{case3 i}
    \end{equation}
    Observe that every index on the simple roots appearing in the input of Kostant's partition function in \Cref{case3 i} have the form $i+c$ for some $c\in[j-1]$.
    We shift every index on the simple roots by $-i$ to set the smallest index to 1. Therefore, by \Cref{lemma: rescaling q partition is cool}, \Cref{case3 i} can be expressed as
    \begin{equation}
        -\sum_{\substack{\sigma\in C \\ \sigma=s_{i+j}\sigma_1'}} (-1)^{\sigma_1'}\qpartition{\alpha_{i+1,i+j-1}-\sum_{x\in\mathcal{I}_1}\alpha_x} = - \sum_{\sigma\in C'} (-1)^{\ell(\sigma)}\qpartition{\alpha_{1,j-1}-\sum_{x\in\mathcal{I}_1'}\alpha_x},\label{C'}
    \end{equation}
    where $C'$ and $\mathcal{I}_1'$ denote the sets $C$ and $\mathcal{I}_1$ respectfully after the shift in indices. 
    Moreover, notice that after the shift, $\sigma\in C'$ implies $\sigma = \prod_{x\in\mathcal{I}'_1}s_x$, where $\mathcal{I}'_1$ is a set of nonconsecutive integers in the interval $[2,j-2]$.

    Now consider $m_q(\hr_{j-1},0)$ in rank $j-1$. By \Cref{pamela alternation result}, $\sigma\in\mathcal{A}_{j-1}(\hr_{j-1},0)$ if and only if $\sigma=\prod_{x\in\mathcal{I}}s_x$, where $\mathcal{I}$ is a set of nonconsecutive integers in the interval $[2,j-2]$. Therefore, by \Cref{lemma: alphas get spit out}, we can express $m_q(\hr_{j-1},0)$ as
    \begin{align}
        m_q(\hr_{j-1},0) &= \sum_{\sigma\in W}(-1)^{\ell(\sigma)}\qpartition{\sigma(\hr_{j-1}+\rho_{j-1})-\rho_{j-1}}\notag \\
        &= \sum_{\sigma\in\mathcal{A}_{j-1}(\hr_{j-1},0)}(-1)^{\ell(\sigma)}\qpartition{\hr_{j-1}-\sum_{x\in\mathcal{I}}\alpha_x}.\label{zero weight mult for C'}
    \end{align}
    Now we show that \Cref{C'} is equivalent to \Cref{zero weight mult for C'}, despite these expressions being in different ranks. To prove this, we first show that the sum contains an equivalent number of terms to the size of the alternation set $\mathcal{A}_{j-1}(\hr_{j-1},0)$, and that every term in \Cref{C'} has a unique corresponding term in \Cref{zero weight mult for C'}, which proves the claim. First, recall that by \Cref{pamela alternation result},
    \begin{equation*}
        |\mathcal{A}_{j-1}(\hr_{j-1},0)|=F_{j-1},
    \end{equation*}
    and by \Cref{thm:whats in the alt set}, if $\sigma\in C'$ then $\sigma=\prod_{x\in\mathcal{I}_1'}s_x$, where $\mathcal{I}_1'$ is a set of nonconsecutive integers in the interval $[2,j-2]$. Notice that this is exactly equivalent to choosing nonconsecutive integers in the interval $[j-3]$, which is counted by $F_{j-1}$. This implies $|C'|=F_{j-1}$, and hence both sums have an equivalent number of terms. Now we show that every term in \Cref{C'} corresponds uniquely to a term in \Cref{zero weight mult for C'}. Consider the term indexed by the identity element in \Cref{C'}; the corresponding term is of the form $(-1)^{\ell(1)}\qpartition{\alpha_{1,j-1}-0}$. By \Cref{lemma: alphas get spit out},
    \begin{equation*}
        (-1)^{\ell(1)}\qpartition{\alpha_{1,j-1}-0} = \qpartition{\alpha_{1,j-1}} = q(q+1)^{j-2}.
    \end{equation*}
    Likewise, the term indexed by the identity element in $m_q(\hr_{j-1},0)$ is 
    \begin{equation*}
        (-1)^{\ell(1)}\qpartition{1(\hr_{j-1}+\rho_{j-1}-\rho_{j-1})}=\qpartition{\alpha_{1,j-1}} = q(q+1)^{j-2}.
    \end{equation*}
    Therefore, the terms indexed by the identity elements are equal. Now consider any arbitrary non-identity element $\sigma\in C'$ with corresponding term in \Cref{C'} given by
    \begin{equation}
        (-1)^{\ell(\sigma)}\qpartition{\alpha_{1,j-1}-\sum_{x\in\mathcal{I}_1'}\alpha_x},\notag
    \end{equation}
    with $\mathcal{I}_1'$ a set of nonconsecutive integers in the interval $[2,j-2]$. By \Cref{pamela alternation result}, this is equivalent to a unique $\sigma^*\in\mathcal{A}_{j-1}(\hr_{j-1},0)$, using the same indices in $\mathcal{I}_1'$ for the simple reflections as in $\sigma$, but simple reflections in the Weyl group in rank $j-1$. This implies 
    \begin{equation*}
        (-1)^{\ell(\sigma^*)}\qpartition{\alpha_{1,j-1}-\sum_{x\in\mathcal{I}}\alpha_x}=(-1)^{\ell(\sigma)}\qpartition{\alpha_{1,j-1}-\sum_{x\in\mathcal{I}_1'}\alpha_x}.
    \end{equation*}
    Since our choice of non-identity element $\sigma\in C'$ was arbitrary, every term in \Cref{C'} corresponds to a unique term in $m_q(\hr_{j-1},0)$, and the two equations are equivalent. Therefore, by \Cref{eq: pamela result}, 
    \begin{equation}
        -\sum_{\sigma\in C'}(-1)^{\ell(\sigma_1')}\qpartition{\alpha_{1,j-1}-\sum_{x\in\mathcal{I}_1'}\alpha_x} = -m_q(\hr_{j-1},0) = - \sum_{t=1}^{j-1} q^t.\label{DIFF OF EQS SUM 3}
    \end{equation}

    Consider the sum $\sum_{\sigma\in D}(-1)^{\ell(\sigma)}\qpartition{\sigma(\hr_r+\rho_r)-\rho_r-\alpha_{1,i}-\alpha_{i+j+1,r}}$ in \Cref{turn this into pamela result}.  
    For any arbitrary $\sigma=s_{i+1}s_{i+j}\prod_{x\in\mathcal{I}_1}s_x$, the input of Kostant's partition function can be reduced to a partition function on a positive root minus a sum of simple roots. By \Cref{lemma: alphas get spit out}, for any $s_{i+1}s_{i+j}\prod_{x\in\mathcal{I}_1}s_x$ with $\mathcal{I}_1$ a set of nonconsecutive integers in the interval $[i+3,i+j-2]$, we have that
    \begin{align}
        &(-1)^{\ell(\sigma)}\qpartition{\sigma(\hr_r+\rho_r)-\rho_r-\alpha_{1,i}-\alpha_{i+j+1,r}} \notag\\
        &\hspace{4cm}= (-1)^{\ell(s_{i+1}s_{i+j}\sigma_1')}\qpartition{\hr_r-\alpha_{1,i}-\alpha_{i+1}-\alpha_{i+j}-\alpha_{i+j+1,r}-\sum_{x\in\mathcal{I}_1}\alpha_x} \notag\\
        &\hspace{4cm}= (-1)^{\ell(s_{i+1}s_{i+j}\sigma_1')}\qpartition{\alpha_{i+2,i+j-1}-\sum_{x\in\mathcal{I}_1}\alpha_x}.\notag
    \end{align}
    Note that $\sum_{x\in\mathcal{I}_1}\alpha_x=0$ if $\sigma_1'=1$.
    Since $\ell(s_{i+1}s_{i+j}\sigma_1')=\ell(s_{i+1})+\ell(s_{i+j}) + \ell(\sigma_1')=2+\ell(\sigma_1')$, the sum corresponding to the set $D$ can be expressed as,
    \begin{equation}
         \sum_{\sigma\in D} (-1)^{\ell(\sigma)}\qpartition{\alpha_{i+2,i+j-1}-\sum_{x\in\mathcal{I}_1}\alpha_x} = \sum_{\substack{\sigma\in D \\ \sigma=s_{i+1}s_{i+j}\sigma_1'}}(-1)^{\ell(\sigma_1')}\qpartition{\alpha_{i+2,i+j-1}-\sum_{x\in\mathcal{I}_1}\alpha_x}.\label{D'} 
    \end{equation}
    Observe that every index on the simple roots appearing as an input in \Cref{D'} has the form $i+c$ for some $c\in[2,j-1]$. We shift the indices by $-i-1$ to set the smallest index to 1. Therefore, by \Cref{lemma: rescaling q partition is cool}, \Cref{D'} can be expressed as
    \begin{equation}
        \sum_{\substack{\sigma\in D \\ \sigma=s_{i+1}s_{i+j}\sigma_1'}} (-1)^{\ell(\sigma_1')}\qpartition{\alpha_{i+2,i+j-1}-\sum_{x\in\mathcal{I}_1}\alpha_x} = \sum_{\sigma\in D'}(-1)^{\ell(\sigma)}\qpartition{\alpha_{1,j-2}-\sum_{x\in\mathcal{I}'_1}\alpha_x},\notag
    \end{equation}
    where $D'$ and $\mathcal{I}_1'$ denote the sets $D$ and $\mathcal{I}_1$ after the shift in indices. Moreover, notice that after the shift, $\sigma\in D'$ implies $\sigma=\prod_{x\in\mathcal{I}_1'}s_x$, where $\mathcal{I}_1'$ is a set of nonconsecutive integers in the interval $[2,j-3]$.

    Now consider $m_q(\hr_{j-2},0)$ in rank $j-2$. By \Cref{pamela alternation result}, $\sigma\in\mathcal{A}_{j-2}(\hr_{j-2},0)$ if and only if $\sigma=\prod_{x\in\mathcal{I}}s_x$, where $\mathcal{I}$ is a set of nonconsecutive integers in the interval $[2,j-2]$. Therefore, by \Cref{lemma: alphas get spit out}, we can express $m_q(\hr_{j-2},0)$ as
    \begin{align}
        m_q(\hr_{j-2},0) &= \sum_{\sigma \in W}(-1)^{\ell(\sigma)}\qpartition{\sigma(\hr_{j-2}+\rho_{j-2})-\rho_{j-2}} \notag \\
        &= \sum_{\sigma\in\mathcal{A}_{j-2}(\hr_{j-2},0)}(-1)^{\ell(\sigma)}\qpartition{\hr_{j-2}-\sum_{x\in\mathcal{I}}\alpha_x}\label{zero weight for D'}.
    \end{align}
    Now we show that \Cref{D'} is equivalent to \Cref{zero weight for D'}, despite these expressions being in different ranks. To prove this, we first show that the sum contains an equivalent number of terms to the size of the alternation set $\mathcal{A}_{j-2}(\hr_{j-2},0)$, and that every term in \Cref{D'} corresponds uniquely to a term in \Cref{zero weight for D'}. First consider the term indexed by the identity element in \Cref{D'}; the corresponding term is of the form $(-1)^{\ell(1)}\qpartition{\alpha_{1,j-2}-0}$. By \Cref{lemma: alphas get spit out},
    \begin{equation*}
        (-1)^{\ell(1)}\qpartition{\alpha_{1,j-2}-0}=\qpartition{\alpha_{1,j-2}}=q(q+1)^{j-3}.
    \end{equation*}
    Likewise, the term indexed by the identity element in $m_q(\hr_{j-2},0)$ is
    \begin{equation*}
        (-1)^{\ell(1)}\qpartition{\hr_{j-2}}=q(q+1)^{j-3}.
    \end{equation*}
    Therefore the terms indexed by the identity element are equal. Now consider any arbitrary non-identity element $\sigma\in D'$ with corresponding term in \Cref{D'} given by,
    \begin{equation*}
        (-1)^{\ell(\sigma)}\qpartition{\alpha_{1,j-2}-\sum_{x\in\mathcal{I}_1'}\alpha_x},
    \end{equation*}
    with $\mathcal{I}_1'$ a set of nonconsecutive integers in the interval $[2,j-3]$. By \Cref{pamela alternation result}, this is equivalent to a unique $\sigma^*\in\mathcal{A}_{j-2}(\hr_{j-2},0)$, using the same indices in $\mathcal{I}_1'$ for the simple reflections as in $\sigma$, but simple reflections in the Weyl group in rank $j-2$. This implies
    \begin{equation*}
        (-1)^{\ell(\sigma^*)}\qpartition{\hr_{j-2}-\sum_{x\in\mathcal{I}}\alpha_x}=(-1)^{\ell(\sigma)}\qpartition{\alpha_{1,j-2}-\sum_{x\in\mathcal{I}_1'}\alpha_x}.
    \end{equation*}
    Since our choice of non-identity element $\sigma\in D'$ was arbitrary, every term in \Cref{D'} corresponds to a unique term in $m_q(\hr_{j-2},0)$, and the two equations are equivalent. Therefore, by \Cref{eq: pamela result},
    \begin{equation}
        \sum_{\sigma\in D'}(-1)^{\ell(\sigma_1')}\qpartition{\alpha_{1,j-2}-\sum_{x\in\mathcal{I}'_1}\alpha_x} = m_q(\hr_{j-2},0) = \sum_{t=1}^{j-2}q^t.\label{DIFF OF EQS SUM 4}
    \end{equation}
    We are now ready to prove the claim. By substituting \Cref{DIFF OF EQS SUM 1}, \Cref{DIFF OF EQS SUM 2}, \Cref{DIFF OF EQS SUM 3}, and \Cref{DIFF OF EQS SUM 4} into \Cref{turn this into pamela result}, the weight $q$-multiplicity $m_q(\hr_r,\alpha_{1,i}+\alpha_{i+j+1,r})$ can be expressed as,
    \begin{align}
         m_q(\hr,\alpha_{1,i}+\alpha_{i+j+1,r}) &= \sum_{t=1}^j q^t -2\sum_{t=1}^{j-1} q^t + \sum_{t=1}^{j-2} q^t\notag \\
         &=q^j+q^{j-1}-2q^{j-1}+\sum_{t=1}^{j-2}\left(q^t-2q^t+q^t\right) \notag\\
         &=q^j-q^{j-1},\notag
    \end{align}
    for any $j\geq 3$, 
    which completes the proof.
\end{proof}
\section{The $q$-analog of Kostant's Weight Multiplicity Formula for sums of distinct simple roots }\label{section: 4 main results}
In this section, we let $I$ be any nonempty subset of $[r]$. We begin by showing that the terms of the weight $q$-multiplicity $m_q(\hr,\alpha_I)$ can be expressed as a product of sums similar to that of \Cref{phase 1 turn into rank j result}. 
The key detail in the proof of \Cref{thm: difference of q's} was converting the problem into a sum of weight $q$-multiplicities in lower ranks whose outputs are known results.  We employ a similar strategy for computing $m_q(\hr,\alpha_I)$ for any nonempty indexing set $I\subseteq [r]$.  
Specifically, we convert $m_q(\hr,\alpha_I)$ into a product of weight $q$-multiplicities  described in \Cref{thm: difference of q's} and \Cref{kim result} in lower ranks. The necessity of \Cref{kim result} is entirely dependent on the existence of $1$ or $r$ in the indexing set $I$. 
Throughout this section, we consider multiple ranks at one time. To avoid confusion, we decorate $\hr$, $\rho$, and $\mathcal{A}$ with a subscript to denote the rank.
Now we provide two technical lemmas that allows us to express $m_q(\hr,\alpha_I)$ as a product of weight $q$-multiplicities.
\begin{lemma}\label{lemma: mu a sum of positive roots turns KWMF into a product}
    Let $\hr$ denote the highest root in $\speclin$. Fix a nonempty indexing set $I\subseteq [r]$ with interval partition $I=\bigsqcup_{x=1}^{n(I)}[i_x,j_x]$. If $I^c=  
    \bigsqcup_{x=1}^{n(I^c)}[i'_x,j'_x]$ with $n(I^c)$ as in \Cref{def:n(I^c)}, then 
    \begin{equation}\label{eq:KWMF into a product}
        m_q(\hr_r,\alpha_I)=\prod_{x=1}^{n(I^c)}\sum_{\substack{\sigma\in\mathcal{A}(\hr,\alpha_I) \\ \sigma=\sigma_x}}(-1)^{\ell(\sigma_x)}\qpartition{\alpha_{i'_x,j'_x}-\sum_{i\in\mathcal{I}_x}\alpha_i},
    \end{equation}
    where, for each $x\in [n(I^c)]$, $\mathcal{I}_x=\{x_1,x_2,\ldots,x_m\}$ is a (possibly empty) set of nonconsecutive integers in the interval $[i'_x,j'_x]$, $\sigma_x=\prod_{i\in \mathcal{I}_x}s_i$, and we set $\sum_{i\in\mathcal{I}_x}\alpha_i=0$ if $\sigma_x=1$.  
\end{lemma}
\begin{proof}
    By \Cref{thm:whats in the alt set}, $\sigma\in\thealtset$ if and only if $\sigma$ is of the form
    \[\sigma=1 \quad \text{ and } \quad \sigma= \sigma_1\sigma_2\ldots\sigma_{n(I^c)},\]
    where, for each $x\in [n(I^c)]$, $\mathcal{I}_x$ is a (possibly empty) set of nonconsecutive integers in the interval $[i_x',j_x']$, and $\sigma_x=\prod_{i\in\mathcal{I}_x}s_i$. 
    Let $\sigma=\sigma_1\sigma_2\cdots\sigma_{n(I^c)}$ be any arbitrary element of the Weyl alternation set, and consider Kostant's partition function in the corresponding term in $m_q(\hr,\alpha_I)$:
    \begin{equation}
        \qpartition{\sigma(\hr_r+\rho_r)-\rho_r-\alpha_I} = \qpartition{\hr_r - \alpha_{I}-\sum_{x\in[n(I^c)]}\sum_{i\in\mathcal{I}_x}\alpha_i} = \qpartition{\alpha_{I^c}-\sum_{x\in[n(I^c)]}\sum_{i\in\mathcal{I}_x}\alpha_i}.\notag
    \end{equation}
    Since $I^c = \bigsqcup_{x=1}^{n(I^c)}[i'_x,j'_x]$, 
    \begin{align}
        &\qpartition{\alpha_{I^c}-\sum_{x\in[n(I^c)]}\sum_{i\in\mathcal{I}_x}\alpha_i}\notag\\
        &\qquad \qquad = \qpartition{\left(\alpha_{i'_1,j'_1}-\sum_{i\in\mathcal{I}_1}\alpha_i\right)+\left(\alpha_{i'_2,j'_2}-\sum_{i\in\mathcal{I}_2}\alpha_i\right)+\cdots+\left(\alpha_{i'_{n(I^c)},j'_{n(I^c)}}-\sum_{i\in\mathcal{I}_{n(I^c)}}\alpha_i\right)}\notag \\
        &\qquad\qquad = \qpartition{\sum_{x=1}^{n(I^c)}\left(\alpha_{i'_x,j'_x}-\sum_{i\in\mathcal{I}_x}\alpha_i\right)}.\notag
    \end{align}
    By \Cref{lemma: split up partition function},
    \begin{align}
        \qpartition{\sum_{x=1}^{n(I^c)}\left(\alpha_{i'_x,j'_x}-\sum_{i\in\mathcal{I}_x}\alpha_i\right)}
        &=\prod_{x=1}^{n(I^c)}\qpartition{\alpha_{i'_x,j'_x}-\sum_{i\in\mathcal{I}_x}\alpha_i}.\label{break into product 1}
    \end{align}
    Additionally, for any arbitrary $\sigma=\sigma_1\sigma_2\cdots\sigma_{n(I^c)}$, since $\sigma_x$ contains simple roots whose indices are disjoint from every other simple root in $\sigma_y$, then
    \begin{equation}
        \ell(\sigma) = \ell(\sigma_1\sigma_2\cdots \sigma_{n(I^c)}) = \sum_{x=1}^{n(I^c)} \ell(\sigma_x),\notag
    \end{equation}
    which implies that
    \begin{equation}
        (-1)^{\ell(\sigma)} = \prod_{x=1}^{n(I^c)}(-1)^{\ell(\sigma_x)}.\label{negative one stuff}
    \end{equation}
    Therefore, since our choice of term in $m_q(\hr,\alpha_I)$ was arbitrary, by \Cref{break into product 1} and \Cref{negative one stuff}, the weight $q$-multiplicity satisfies,
    \begin{align}
        m_q(\hr_r,\alpha_I) &= \sum_{\substack{\sigma\in\thealtset \\ \sigma=\sigma_1\sigma_2\cdots \sigma_{n(I^c)}}}\prod_{x=1}^{n(I^c)}(-1)^{\ell(\sigma_x)}\qpartition{\alpha_{i'_x,j'_x}-\sum_{i\in\mathcal{I}_x}\alpha_i}.\label{break into product 2}
    \end{align}
    Notice that \Cref{break into product 2} is a sum of terms, each of the form,
    \[(-1)^{\ell(\sigma_1)}\qpartition{\alpha_{i'_1,j'_1}-\sum_{i\in\mathcal{I}_1}\alpha_i}\cdot(-1)^{\ell(\sigma_2)}\qpartition{\alpha_{i'_2,j'_2}-\sum_{i\in\mathcal{I}_2}\alpha_i} \cdots (-1)^{\ell(\sigma_{n(I^c)})}\qpartition{\alpha_{i'_{n(I^c)},j'_{n(I^c)}}-\sum_{i\in\mathcal{I}_{n(I^c)}}\alpha_i} \]
    corresponding to some element of the Weyl alternation set $\sigma=\sigma_1\sigma_2\cdots\sigma_{n(I^c)}$. Note that for any $x\in[n(I^c)]$, we set $\sum_{i\in\mathcal{I}_x}\alpha_i=0$ if $\sigma_x=1$. To complete the proof, we show that by grouping terms together through factoring common $\sigma_x$ factors, we can state $m_q(\hr_r,\alpha_I)$ as a product of sums in the desired form. 
    
    Let $m\in[n(I^c)]$, and consider $\sigma_m=\prod_{i\in\mathcal{I}_m}s_i$ where $\mathcal{I}_m$ is a (possibly empty) fixed collection of integers in the interval $[i'_m,j'_m]$. Consider every term in \Cref{break into product 2} containing a factor corresponding to $\sigma_m$. 
    By factoring, the collection of these terms can be expressed as
    \begin{equation}
        (-1)^{\ell(\sigma_m)}\qpartition{\alpha_{i'_m,j'_m}-\sum_{i\in\mathcal{I}_m}\alpha_i}\cdot \sum_{\substack{\sigma\in\thealtset\\ \sigma=\sigma_1\sigma_2\cdots\sigma_{m-1}\sigma_{m+1}\cdots\sigma_{n(I^c)}}}\prod_{\substack{x=1\\x\neq m}}^{n(I^c)} (-1)^{\ell(\sigma_x)}\qpartition{\alpha_{i'_x,j'_x}-\sum_{i\in\mathcal{I}_x}\alpha_i}.
    \end{equation}
    Therefore, by repeating this factorization for every collection of nonconsecutive integers $\mathcal{I}_m$ in the interval $[i'_m,j'_m]$, \Cref{break into product 2} can be expressed as
    \begin{align}
        &\sum_{\substack{\sigma\in\mathcal{A}(\hr,\alpha_I) \\ \sigma=\sigma_1\sigma_2\cdots\sigma_n}}\prod_{x=1}^{n(I^c)}(-1)^{\ell(\sigma_x
        )}\qpartition{\alpha_{i'_x,j'_x}-\sum_{i\in\mathcal{I}_x}\alpha_i}\notag \\
        &= \sum_{\substack{\sigma\in\thealtset\\ \sigma=\sigma_m}}\left((-1)^{\ell(\sigma_m)}\qpartition{\alpha_{i'_m,j'_m}-\sum_{i\in\mathcal{I}_m}\alpha_i}\cdot \sum_{\substack{\sigma\in\thealtset\\ \sigma=\sigma_1\sigma_2\cdots \sigma_{m-1}\sigma_{m+1}\cdots\sigma_n}}\prod_{\substack{x=1\\x\neq m}}^{n(I^c)} (-1)^{\ell(\sigma_x)}\qpartition{\alpha_{i'_x,j'_x}-\sum_{i\in\mathcal{I}_x}\alpha_i}\right).\label{eq: KWMF with sigma 1 factored out}
    \end{align}
    In \Cref{eq: KWMF with sigma 1 factored out}, we remove the common $\sigma_m$ factor from each grouping of terms, and express the equation as
    \begin{align}
        &\sum_{\substack{\sigma\in\thealtset\\ \sigma=\sigma_m}}\left((-1)^{\ell(\sigma_m)}\qpartition{\alpha_{i'_m,j'_m}-\sum_{i\in\mathcal{I}_m}\alpha_i}\cdot \sum_{\substack{\sigma\in\thealtset\\ \sigma=\sigma_1\sigma_2\cdots \sigma_{m-1}\sigma_{m+1}\cdots\sigma_n}}\prod_{\substack{x=1\\x\neq m}}^{n(I^c)} (-1)^{\ell(\sigma_x)}\qpartition{\alpha_{i'_x,j'_x}-\sum_{i\in\mathcal{I}_x}\alpha_i}\right)\notag \\
        &\qquad\qquad=\left(\sum_{\substack{\sigma\in\thealtset\\ \sigma=\sigma_m}}(-1)^{\ell(\sigma_m)}\qpartition{\alpha_{i'_m,j'_m}-\sum_{i\in\mathcal{I}_m}\alpha_i}\right)\notag\\
        &\hspace{4cm}\cdot\left(\sum_{\substack{\sigma\in\thealtset\\ \sigma=\sigma_1\cdots \sigma_{m-1}\sigma_{m+1}\cdots\sigma_n}}\prod_{\substack{x=1\\x\neq m}}^{n(I^c)} (-1)^{\ell(\sigma_x)}\qpartition{\alpha_{i'_x,j'_x}-\sum_{i\in\mathcal{I}_x}\alpha_i}\right).\label{break into product 3}
    \end{align}
    Notice that by \Cref{break into product 3}, we have expressed $m_q(\hr_r,\alpha_I)$ as a product, with every term corresponding to some $\mathcal{I}_m$ contained in a single factor.
    By repeating this factorization for every remaining $i\in[n(I^c)]$, the weight $q$-multiplicity can be stated as
    \[m_q(\hr,\alpha_I)=\prod_{x=1}^{n(I^c)}\sum_{\substack{\sigma\in\mathcal{A}(\hr,\alpha_I) \\ \sigma=\sigma_x}}(-1)^{\ell(\sigma_x)}\qpartition{\alpha_{i'_x,j'_x}-\sum_{i\in\mathcal{I}_x}\alpha_i},\]
    which completes the proof.
\end{proof}
Now we show that $m_q(\hr,\alpha_I)$, written in the product form as described in \Cref{lemma: mu a sum of positive roots turns KWMF into a product}, can be expressed as a product weight $q$-multiplicities.
\begin{theorem}\label{THE BEST THEOREM EVER}
    Let $\hr$ be the highest root in $\speclin$, and let $I\subseteq [r]$ be nonempty. Let $I^c = I_1 \sqcup I_2 \sqcup \cdots \sqcup I_{n(I^c)}$ as defined in \Cref{def: interval partition}. 
    \begin{itemize}
        \item If $I_1=[1,i_1-1]$ and $\sigma_1=\prod_{i\in\mathcal{I}_1}$ with $\mathcal{I}_1$ a set of nonconsecutive integers in the interval $I_1$, then
        \begin{equation}
            \sum_{\substack{\sigma\in\mathcal{A}_r(\hr_r,\alpha_I)\\\sigma=\sigma_1}}(-1)^{\ell(\sigma_1)}\qpartition{\alpha_{1,i_1-1}-\sum_{i\in\mathcal{I}_1}\alpha_i}=m_q(\hr_{i_1},\alpha_{i_1}).\label{pays the rent left edge case}
        \end{equation}
        \item For any $x\in[n(I^c)]$, if $I_x=[j_x+1,i_{x+1}-1]$ and $\sigma_x=\prod_{i\in\mathcal{I}_x}$ with $\mathcal{I}_x$ a set of nonconsecutive integers in the interval $I_x$, then
        \begin{equation}
        \sum_{\substack{\sigma\in\mathcal{A}_r(\hr_r,\alpha_I)\\\sigma=\sigma_x}}(-1)^{\ell(\sigma_x)}\qpartition{\alpha_{j_x+1,i_{x+1}-1}-\sum_{i\in\mathcal{I}_x}\alpha_i}=m_q(\hr_{i_{x+1}-j_x+1},\alpha_1+\alpha_{i_{x+1}-j_x+1}).\label{pays the rent middle cases}
        \end{equation}
        \item If $I_{n(I^c)}=[j_{n(I)}+1,r]$ and $\sigma_{n(I^c)}=\prod_{i\in\mathcal{I}_{n(I^c)}}s_i$ with $\mathcal{I}_{n(I^c)}$ a set of nonconsecutive integers in the interval $I_{n(I^c)}$, then 
        \begin{equation}
            \sum_{\substack{\sigma\in\mathcal{A}_r(\hr_r,\alpha_I)\\\sigma=\sigma_{n(I)}}}(-1)^{\ell(\sigma_{n(I)})}\qpartition{\alpha_{j_{n(I)}+1,r}-\sum_{i\in\mathcal{I}_{n(I)}}\alpha_i}=m_q(\hr_{r-j_{n(I)}+1},\alpha_1).\label{pays the rent right edge case}
        \end{equation}
    \end{itemize}
\end{theorem}
\begin{proof}
    Fix a nonempty $I\subseteq [r]$. 
    Suppose $I^c$ contains the interval $I_1=[1,i_1-1]$.
    To prove the claim in \Cref{pays the rent left edge case}, we show that the sum contains an equivalent number of terms to the size of the alternation set $\mathcal{A}_{i_1}(\hr_{i_1},\alpha_{i_1})$. Then we show that every term in the sum corresponds to a unique term with nonzero output in the weight $q$-multiplicity, thus showing equivalency despite being in different ranks.
    
    Consider the weight $q$-multiplicity $m_q(\hr_{i_1},\alpha_{i_1})$ in rank $i_1$. 
    By \Cref{kim alternation result}, $\sigma\in\mathcal{A}_{i_1}(\hr_{i_1},\alpha_{i_1})$ if and only if $\sigma=1$ or $\sigma=\prod_{i\in\mathcal{I}} s_i$, where $\mathcal{I}$ is a  collection of nonconsecutive integers in the interval $[2,i_1-1]$, and there are $F_{i_1}$ many such subsets. Moreover, by \Cref{kim alternation result} and \Cref{lemma: alphas get spit out}, $m_q(\hr_{i_1},\alpha_{i_1})$ can be expressed as
    \begin{align}
        m_q(\hr_{i_1},\alpha_{i_1}) &= \sum_{\sigma\in W} (-1)^{\ell(\sigma)}\qpartition{\sigma(\hr_{i_1}+\rho_{i_1})-\rho_{i_1}-\alpha_{i_1}} \notag\\
        &= \sum_{\sigma\in\mathcal{A}_{i_1}(\hr_{i_1},\alpha_{i_1})}(-1)^{\ell(\sigma)}\qpartition{\alpha_{1,i_1-1}-\sum_{i\in\mathcal{I}}\alpha_i}.\label{CASE 2 KIM RESULT IN THE RIGHT FORM}
    \end{align}
    Recall that for $\sigma_1\in\thealtset$, $\sigma_1=1$ or $\sigma_1=\prod_{i\in\mathcal{I}_1}s_i$, where $\mathcal{I}_1$ is a collection of nonconsecutive integers in the interval $[2,i_1-1]$, and there are $F_{i_1}$ many such subsets. Therefore, the sum in \Cref{pays the rent left edge case} 
    contains a number of terms equivalent to the size of the alternation set $\mathcal{A}_{i_1}(\hr_{i_1},\alpha_{i_1})$.

    Now we show that every term in the sum  corresponds uniquely to a term in \Cref{CASE 2 KIM RESULT IN THE RIGHT FORM}. First consider the term indexed by the identity element in the sum. By \Cref{lemma: rescaling q partition is cool},
    \begin{equation}
        (-1)^{\ell(1)}\qpartition{\alpha_{1,i_1-1}} = q(q+1)^{i_1-2}.\notag
    \end{equation}
    Likewise, the term indexed by the identity element in $m_q(\hr_{i_1},\alpha_{i_1})$ is also of the form
    \begin{equation}
        (-1)^{\ell(1)}\qpartition{\alpha_{1,i_1-1}}=q(q+1)^{i_1-2},\notag
    \end{equation}
    hence the terms indexed by the identity elements are of equal value. Now let $\sigma=\sigma_1$ such that $\ell(\sigma_1)>0$. The corresponding term in the sum is given by,
    \begin{equation}
        (-1)^{\ell(\sigma_1)}\qpartition{\alpha_{1,i_1-1}-\sum_{i\in\mathcal{I}_1}\alpha_i},\notag
    \end{equation}
    with $\mathcal{I}_1$ a set of nonconsecutive integers in the interval $[2,i_1-1]$. By \Cref{kim alternation result}, this is equivalent to a unique $\sigma^*\in\mathcal{A}_{i_1}(\hr_{i_1},\alpha_{i_1})$ using the same indices in $\mathcal{I}_1$ for the simple reflections as in $\sigma_1$, but simple reflections in the Weyl group in rank $i_1$. This implies 
    \begin{align}
        (-1)^{\ell(\sigma^*)}\qpartition{\alpha_{1,i_1-1}-\sum_{i\in\mathcal{I}}\alpha_i} = (-1)^{\ell(\sigma_1)}\qpartition{\alpha_{1,i_1-1}-\sum_{i\in\mathcal{I}_1}\alpha_i}.\notag
    \end{align}
    Since our choice of non-identity element was arbitrary, indeed every term in the sum corresponds to a unique term with nonzero output in $m_q(\hr_{i_1},\alpha_{i_1})$. Therefore, 
    \begin{equation}
        \sum_{\substack{\sigma\in\thealtset\\\sigma=\sigma_1}}(-1)^{\ell(\sigma_1)}\qpartition{\alpha_{1,i_1-1}-\sum_{i\in\mathcal{I}_1}\alpha_i}= m_q(\hr_{i_1},\alpha_{i_1}),\notag
    \end{equation}
    which proves the claim. 
    Now suppose $I^c$ contains the interval $[j_x+1,i_{x+1}-1]$ for any $x\in[n(I^c)]$.
    To prove the claim in \Cref{pays the rent middle cases}, we show that the sum contains an equivalent number of terms to the size of the alternation set $\mathcal{A}_r(\hr_{i_{x+1}-j_x+1},\alpha_1+\alpha_{i_{x+1}-j_x+1})$. 
    Then we show that every term in the sum corresponds to a unique term with nonzero output in the weight $q$-multiplicity, thus showing equivalency despite being in different ranks.
    
    Consider the weight $q$-multiplicity $m_q(\hr_{i_{x+1}-j_x+1},\alpha_1+\alpha_{i_{x+1}-j_x+1})$ in rank $i_{x+1}-j_x+1$.
    By \Cref{thm:whats in the alt set}, $\sigma\in\mathcal{A}_r(\hr_{i_{x+1}-j_x+1},\alpha_1+\alpha_{i_{x+1}-j_x+1})$ if and only if $\sigma=1$ or $\sigma=\prod_{i\in\mathcal{I}}s_i$ where $\mathcal{I}$ is a set of nonconsecutive integers in the interval $[2,i_{x+1}-j_x]$, and there are $F_{i_{x+1}-j_x+1}$ many such subsets. Moreover, by \Cref{lemma: alphas get spit out} and \Cref{thm:whats in the alt set}, the weight $q$-multiplicity can be expressed as 
    \begin{align}
        &m_q(\hr_{i_{x+1}-j_x+1},\alpha_1+\alpha_{i_{x+1}-j_x+1}) \notag\\
        &\qquad \qquad= \sum_{\sigma\in W}(-1)^{\ell(\sigma)}\qpartition{\sigma(\hr_{i_{x+1}-j_x+1}+\rho_{i_{x+1}-j_x+1})-\rho_{i_{x+1}-j_x+1}-\alpha_1-\alpha_{i_{x+1}-j_x+1}}\notag \\
        &\qquad \qquad= \sum_{\sigma\in\mathcal{A}_{i_{x+1}-j_x+1}(\hr_{i_{x+1}-j_x+1},\alpha_1+\alpha_{i_{x+1}-j_x+1})} (-1)^{\ell(\sigma)}\qpartition{\alpha_{2,i_{x+1}-j_x}-\sum_{x\in\mathcal{I}}\alpha_i}.\label{weight q mult in pay the rent}
    \end{align} 
    Recall that for $\sigma_x\in\thealtset$, $\sigma_x=1$ or $\sigma_x=\prod_{i\in\mathcal{I}_x}s_x$ where $\mathcal{I}_x$ is a collection of nonconsecutive integers in the interval $I_x=[j_x+1,i_{x+1}-1]$, and there are $F_{i_{x+1}-j_x+1}$ many such subsets. 
    Therefore the sum in \Cref{pays the rent middle cases} contains a number of terms equivalent to the size of the alternation set $\mathcal{A}_{i_{x+1}-j_x+1}(\hr_{i_{x+1}-j_x+1},\alpha_1+\alpha_{i_{x+1}-j_x+1})$.

    For the sum, we shift the indices on the simple roots appearing in the input of Kostant's partition function by $-j_x+1$ to set the smallest index to $2$. Hence the sum in \Cref{pays the rent middle cases} can be expressed as,
    \begin{equation}
        \sum_{\substack{\sigma\in\mathcal{A}_r(\hr_r,\alpha_I)\\\sigma=\sigma_x}}(-1)^{\ell(\sigma_x)}\qpartition{\alpha_{j_x+1,i_{x+1}-1}-\sum_{i\in\mathcal{I}_x}\alpha_i}=\sum_{\substack{\sigma\in\mathcal{A}_r(\hr_r,\alpha_I)\\\sigma=\sigma_x}}(-1)^{\ell(\sigma_x)}\qpartition{\alpha_{2,i_{x+1}-j_x}-\sum_{i\in\mathcal{I}'_x}\alpha_i},\label{silly sum}
    \end{equation}
    where $\mathcal{I}'_x$ denotes the set $\mathcal{I}_x$ after the shift in indices. 
    Now we show that every term in \Cref{silly sum} corresponds uniquely to a term in \Cref{weight q mult in pay the rent}.
    First, consider the term indexed by the identity element in \Cref{silly sum}. By \Cref{lemma: rescaling q partition is cool},
    \begin{equation}
        (-1)^{\ell(1)}\qpartition{\alpha_{2,i_{x+1}-j_x}-0}=\qpartition{\alpha_{2,i_{x+1}-j_x}}=q(q+1)^{i_{x+1}-j_x-2}.\notag
    \end{equation}
    Likewise, the term indexed by the identity element in $m_q(\hr_{i_{x+1}-j_x+1},\alpha_1+\alpha_{i_{x+1}-j_x+1})$ is also of the form 
    \begin{equation}
        (-1)^{\ell(1)}\qpartition{\alpha_{2,i_{x+1}-j_x}-0}=\qpartition{\alpha_{2,i_{x+1}-j_x}}=q(q+1)^{i_{x+1}-j_x-2},\notag
    \end{equation}
    hence the terms indexed by the identity elements are of equal value. Now let $\sigma_x\in\mathcal{A}_r(\hr_r,\alpha_I)$ such that $\ell(\sigma_x)>0$. The corresponding term in \Cref{silly sum} given by,
    \begin{equation}
        (-1)^{\ell(\sigma_x)}\qpartition{\alpha_{2,i_{x+1}-j_x}-\sum_{x\in\mathcal{I}_x'}\alpha_i},\notag
    \end{equation}
    with $\mathcal{I}_x'$ a set of nonconsecutive integers in the interval $[2,i_{x+1}-j_x]$.
    By \Cref{thm:whats in the alt set}, this is equivalent to a unique $\sigma^*\in\mathcal{A}_{i_{x+1}-j_x+1}(\hr_{i_{x+1}-j_x+1},\alpha_1+\alpha_{i_{x+1}-j_x+1})$ using the same indices in $\mathcal{I}'_x$ for the simple reflections as in $\sigma_x$, but simple reflections in the Weyl group in rank $i_{x+1}-j_x+1$. This implies 
    \begin{equation}
        (-1)^{\ell(\sigma^*)}\qpartition{\alpha_{2,i_{x+1}-j_x}-\sum_{x\in\mathcal{I}}\alpha_i} = (-1)^{\ell(\sigma_x)}\qpartition{\alpha_{2,i_{x+1}-j_x}-\sum_{i\in\mathcal{I}'_x}\alpha_i}.\notag
    \end{equation}
    Since our choice of non-identity element was arbitrary, indeed every term in \Cref{silly sum} corresponds to a unique term with nonzero output in \Cref{weight q mult in pay the rent}. Therefore, 
        \begin{equation}
        \sum_{\substack{\sigma\in\mathcal{A}_r(\hr_r,\alpha_I)\\\sigma=\sigma_x}}(-1)^{\ell(\sigma_x)}\qpartition{\alpha_{j_x+1,i_{x+1}-1}-\sum_{i\in\mathcal{I}_x}\alpha_i}=m_q(\hr_{i_{x+1}-j_x+1},\alpha_1+\alpha_{i_{x+1}-j_x+1}),\notag
    \end{equation}
    which proves the claim.
    Now suppose $I^c$ contains the interval $I_{n(I^c)}=[j_{n(I)}+1,r]$. To prove the claim in \Cref{pays the rent right edge case}, we show that the sum contains an equivalent number of terms to the size of the alternation set $\mathcal{A}_{r-j_{n(I)}+1}(\hr_{r-j_{n(I)}+1},\alpha_1)$. Then we show that every term in the sum corresponds to a unique term with nonzero output in the weight $q$-multiplicity, thus showing equivalency despite being in different ranks.

    Consider the weight $q$-multiplicity $m_q(\hr_{r-j_{n(I)}+1},\alpha_1)$ in rank $r-j_{n(I)}+1$. By \Cref{kim alternation result}, 
    \[\sigma\in\mathcal{A}_{r-j_{n(I)}+1}(\hr_{r-j_{n(I)}+1},\alpha_1)\]
    if and only if $\sigma=1$ or $\sigma=\prod_{i\in\mathcal{I}}s_i$ where $\mathcal{I}$ is a collection of nonconsecutive integers in the interval $[2,r-j_{n(I)}]$, and there are $F_{r-j_{n(I)}+1}$ many such subsets. Moveover, by \Cref{kim result} and \Cref{lemma: alphas get spit out}, $m_q(\hr_{r-j_{n(I)}+1},\alpha_1)$ can be expressed as,
    \begin{align}
        m_q(\hr_{r-j_{n(I)}+1},\alpha_1)&=\sum_{\sigma\in W}(-1)^{\ell(\sigma)}\qpartition{\sigma(\hr_{r-j_{n(I)}+1}+\rho_{r-j_{n(I)}+1})-\rho_{r-j_{n(I)}+1}-\alpha_1}\notag \\
        &= \sum_{\sigma\in\mathcal{A}_{r-j_{n(I)}+1}(\hr_{r-j_{n(I)}+1},\alpha_1)}(-1)^{\ell(\sigma)}\qpartition{\alpha_{2,r-j_{n(I)}+1}-\sum_{i\in\mathcal{I}}\alpha_i}.\label{right edge cleaned up}
    \end{align}
    Recall that for $\sigma_{n(I)}\in\thealtset$, $\sigma_{n(I)}=1$ or $\sigma_{n(I)}=\prod_{i\in\mathcal{I}_1}s_i$ where $\mathcal{I}_{n(I)}$ is a collection of nonconsecutive integers in the interval $[j_{n(I)}+1,r-1]$, and there are $F_{r-j_{n(I)}+1}$ many such subsets. Therefore, the sum in \Cref{pays the rent right edge case} contains a number of terms equivalent to  the size of the alternation set $\mathcal{A}_{r-j_{n(I)}+1}(\hr_{r-j_{n(I)}+1},\alpha_1)$.
    
    For the sum, we shift the indices on the simple roots appearing in the input of Kostant's partition function by $-j_{n(I)}+1$ to set the smallest index to $2$. Hence the sum in \Cref{pays the rent right edge case} can be expressed as,
    \begin{equation}
        \sum_{\substack{\sigma\in\mathcal{A}_r(\hr_r,\alpha_I)\\\sigma=\sigma_{n(I)}}}(-1)^{\ell(\sigma_{n(I)})}\qpartition{\alpha_{j_{n(I)}+1,r}-\sum_{i\in\mathcal{I}_{n(I)}}\alpha_i}=\sum_{\substack{\sigma\in\mathcal{A}_r(\hr_r,\alpha_I)\\\sigma=\sigma_{n(I)}}}(-1)^{\ell(\sigma_{n(I)})}\qpartition{\alpha_{2,r-j_{n(I)}+1}-\sum_{i\in\mathcal{I}'_{n(I)}}\alpha_i},\label{last sum cleaned up}
    \end{equation}
    where $\mathcal{I}'_{n(I)}$ denotes the set $\mathcal{I}_{n(I)}$ after the shift in indices.
    Now we show that every term in \Cref{last sum cleaned up} corresponds uniquely to a term in \Cref{right edge cleaned up}. First consider the term indexed by the identity element in the sum. By \Cref{lemma: rescaling q partition is cool},
    \begin{equation*}
        (-1)^{\ell(1)}\qpartition{\alpha_{2,r-j_{n(I)}+1}-0} = \qpartition{\alpha_{2,r-j_{n(I)}+1}}=q(q+1)^{r-j_{n(I)}-1}.
    \end{equation*}
    Likewise, the term indexed by the identity element in $m_q(\hr_{r-j_{n(I)}+1},\alpha_1)$ is also of the form
    \begin{equation*}
        (-1)^{\ell(1)}\qpartition{\alpha_{2,r-j_{n(I)}+1}-0} = \qpartition{\alpha_{2,r-j_{n(I)}+1}}=q(q+1)^{r-j_{n(I)}-1},
    \end{equation*}
    hence the terms indexed by the identity elements are of equal value. Now let $\sigma_{n(I)}\in\thealtset$ such that $\ell(\sigma_{n(I)})>0$. The corresponding term in \Cref{right edge cleaned up} is given by, 
    \begin{equation*}
        (-1)^{\ell(\sigma_{n(I)})}\qpartition{\alpha_{2,r-j_{n(I)}+1}-\sum_{i\in\mathcal{I}'_{n(I)}}\alpha_i},
    \end{equation*}
    with $\mathcal{I}'_{n(I)}$ a set of nonconsecutive integers in the interval $[2,r-j_{n(I)}+1]$. By \Cref{kim result}, this is equivalent to a unique $\sigma^*\in\mathcal{A}_{r-j_{n(I)}+1}(\hr_{r-j_{n(I)}+1},\alpha_1)$ using the same indices in $\mathcal{I}'_{n(I)}$ for the simple reflections as in $\sigma_{n(I)}$, but simple reflections in the Weyl group in rank $r-j_{n(I)}+1$. This implies 
    \begin{equation*}
        (-1)^{\ell(\sigma^*)}\qpartition{\alpha_{2,r-j_{n(I)}+1}-\sum_{i\in\mathcal{I}}\alpha_i}
        =(-1)^{\ell(\sigma_{n(I)})}\qpartition{\alpha_{2,r-j_{n(I)}+1}-\sum_{i\in\mathcal{I}'_{n(I)}}\alpha_i}.
    \end{equation*}
    Since our choice of non-identity element was arbitrary, indeed every term in \Cref{last sum cleaned up} corresponds to a unique term with nonzero output in \Cref{right edge cleaned up}. Therefore,
    \begin{equation*}
        \sum_{\substack{\sigma\in\mathcal{A}_r(\hr_r,\alpha_I)\\\sigma=\sigma_{n(I)}}}(-1)^{\ell(\sigma_{n(I)})}\qpartition{\alpha_{j_{n(I)}+1,r}-\sum_{i\in\mathcal{I}_{n(I)}}\alpha_i}=m_q(\hr_{r-j_{n(I)}+1},\alpha_1),
    \end{equation*}
    which proves the claim and completes the proof.
\end{proof}
At last we show the weight $q$-multiplicity $m_q(\hr,\alpha_I)$ for any nonempty $I\subseteq [r]$.
\THEBIGONE
\begin{proof} 
    Fix a nonempty $I\subseteq [r]$.
    By \Cref{thm:whats in the alt set}, $\sigma\in \thealtset$ if and only if $\sigma$ is of the form,
    \[\sigma=1 \quad \text{ or } \quad \sigma = \sigma_1\sigma_2\cdots \sigma_{n(I^c)}.\]
    For each $x\in [n(I^c)]$, $\mathcal{I}_x$ is a (possibly empty) set of nonconsecutive integers in the interval $[i'_x,j'_x]$, and $\sigma_x=\prod_{i\in\mathcal{I}_x}s_i$. Additionally, by \Cref{def:n(I^c)}, $n(I^c)\in\{n(I)-1,n(I),n(I)+1\}$ depending on the existence of $1$ or $r$ in the indexing set $I$. By \Cref{lemma: mu a sum of positive roots turns KWMF into a product},
    \begin{align}
        m_q(\hr,\alpha_I) &= \prod_{x=1}^{n(I^c)} \sum_{\substack{\sigma\in\thealtset\\ \sigma=\sigma_x}}(-1)^{\ell(\sigma_x)}\qpartition{\alpha_{i'_x,j'_x}-\sum_{i\in\mathcal{I}_x}\alpha_i}.\label{big result starting point}
    \end{align}
    To prove the claim, we manipulate \Cref{big result starting point} into a product of weight $q$-multiplicities in lesser ranks.  
    Then we utilize \Cref{THE BEST THEOREM EVER}, \Cref{thm: difference of q's}, and possibly \Cref{kim result}, to give the value of each weight $q$-multiplicity formula appearing in the product. More precisely, if $1$ or $r$ are not in $I$, then we use \Cref{kim result} in addition to \Cref{thm: difference of q's}. Thus, there are four cases to consider depending on the existence of $1$ or $r$ in the indexing set $I$. 
    \newline\noindent\textbf{Case 1:} Assume both $1,r\in I$: \\
    This implies that $\alpha_I$ contains the terms $\alpha_1$ and $\alpha_r$. Additionally, by \Cref{def:n(I^c)}, $n(I^c)=n(I)-1$ and 
    \begin{equation}
        I^c=\bigsqcup_{x=1}^{n(I)-1} [j_x+1,i_{x+1}-1].\label{interval for case 1}
    \end{equation}
    Therefore, in this case, we can state \Cref{big result starting point} as
    \begin{equation}
        m_q(\hr_r,\alpha_I) = \prod_{x=1}^{n(I)-1}\sum_{\substack{\sigma\in\mathcal{A}_r(\hr,\alpha_I)\\ \sigma=\sigma_x}}(-1)^{\ell(\sigma_x)}\qpartition{\alpha_{j_x+1,i_{x+1}-1}-\sum_{i\in\mathcal{I}_x}\alpha_i}.\label{big theorem case 1 start}
    \end{equation}
    By \Cref{THE BEST THEOREM EVER}, 
    \begin{equation}
        m_q(\hr_r,\alpha_I) = \prod_{x=1}^{n(I)-1}m_q(\hr_{i_{x+1}-j_x+1},\alpha_1+\alpha_{i_{x+1}-j_x+1})\notag
    \end{equation}
    and by \Cref{thm: difference of q's}, 
    \begin{equation}
        m_q(\hr_r,\alpha_I) 
        = \prod_{x=1}^{n(I)-1} q^{i_{x+1}-j_x-1}-q^{i_{x+1}-j_x-2}.\label{big theorem case 1 ALMOST DONE}
    \end{equation}
    We conclude this case by showing that \Cref{big theorem case 1 ALMOST DONE} is equivalent to \Cref{eq: KWMF final poly} as claimed:
    \begin{align}
        m_q(\hr,\alpha_I) &= \prod_{x=1}^{n(I)-1}q^{i_{x+1}-j_x-1}-q^{i_{x+1}-j_x-2} \notag \\
        &= \prod_{x=1}^{n(I)-1}q^{i_{x+1}-j_x-2}(q-1)\notag \\
        &= (q-1)^{n(I)-1}\cdot\prod_{x=1}^{n(I)-1}q^{i_{x+1}-j_x-2}\notag \\
        &= (q-1)^{n(I)-1}\cdot q^{\sum_{x=1}^{n(I)-1}i_{x+1}-j_x-2}\notag \\
        &= (q-1)^{n(I)-1}\cdot q^{-n(I)+1+\sum_{x=1}^{n(I)-1}i_{x+1}-j_x-1}.\notag
    \end{align}
    By \Cref{interval for case 1} and since $r=|I|+|I^c|$,
    \begin{equation}
        |I^c|=
        \sum_{x=1}^{n(I)-1}i_{x+1}-j_x-1,\notag
    \end{equation}
    which implies that
    \begin{equation}
        m_q(\hr_r,\alpha_I) = (q-1)^{n(I)-1}\cdot q^{r-|I|-n(I)+1}\notag
    \end{equation}
    as claimed. 
    \newline\noindent \textbf{Case 2:} Assume $1\notin I$ and $r\in I$. \\
    This implies that $\alpha_I$ contains the term $\alpha_r$, and the term with the smallest index is $\alpha_{i_1}$, with $i_1>1$. Additionally, by \Cref{def:n(I^c)}, $n(I^c)=n(I)$ and 
    \begin{equation}
        I^c = [1,i_1-1] \sqcup \bigsqcup_{x=1}^{n(I)-1}[j_x+1,i_{x+1}-1].\label{interval partition for case 2}
    \end{equation}
    Therefore, in this case, we can state \Cref{big result starting point} as,
    \begin{align}
        &m_q(\hr,\alpha_I) \notag\\
        &=\left(\sum_{\substack{\sigma\in\thealtset\\\sigma=\sigma_1}}(-1)^{\ell(\sigma_1)}\qpartition{\alpha_{1,i_1-1}-\sum_{i\in\mathcal{I}_1}\alpha_i}\right)\cdot \left(\prod_{x=2}^{n(I)}\sum_{\substack{\sigma\in\thealtset\\\sigma=\sigma_x}}(-1)^{\ell(\sigma_x)}\qpartition{\alpha_{j_x+1,i_{x+1}-1}-\sum_{i\in\mathcal{I}_x}\alpha_i}\right).\notag
    \end{align}
    By \Cref{THE BEST THEOREM EVER},
    \begin{align}
        m_q(\hr_r,\alpha_I) = m_q(\hr_{i_1},\alpha_{i_1})\cdot \prod_{x=2}^{n(I)}m_q(\hr_{i_{x+1}-j_x+1},\alpha_1+\alpha_{i_{x+1}-j_x+1}),\notag
    \end{align}
    and by \Cref{thm: difference of q's} and \Cref{kim result},
    \begin{align}
        m_q(\hr_r,\alpha_I) = q^{i_1-1}\cdot \prod_{x=1}^{n(I)-1}q^{i_{x+1}-j_x-1}-q^{i_{x+1}-j_x-2}.\label{Case 2 with best lemma used}
    \end{align}
    We conclude this case by showing \Cref{Case 2 with best lemma used} is equivalent to \Cref{eq: KWMF final poly}:
    \begin{align}
        q^{i_1-1}\cdot \prod_{x=1}^{n(I)-1}q^{i_{x+1}-j_x-1}-q^{i_{x+1}-j_x-2}&=q^{i_1-1}\cdot \prod_{x=1}^{n(I)-1}q^{i_{x+1}-j_x-2}\cdot(q-1) \notag\\
        &= q^{i_1-1}\cdot (q-1)^{n(I)-1}\cdot\prod_{x=1}^{n(I)-1}q^{i_{x+1}-j_x-2}\notag \\
        &= (q-1)^{n(I)-1}\cdot q^{i_1-1}\cdot q^{\sum_{x=1}^{n(I)-1} i_{x+1}-j_x-2}\notag\\
        &=(q-1)^{n(I)-1}\cdot q^{i_1-1-n(I)+1+\sum_{x=1}^{n(I)-1} i_{x+1}-j_x-1}\label{eq: case 2 exp clean up}.
    \end{align}
    Since $r=|I|+|I^c|$, by \Cref{interval partition for case 2},
    \begin{equation}
        |I^c| = i_1-1 + \sum_{x=1}^{n(I)-1} i_{x+1}-j_x-1
    \end{equation}
    which implies that,
    \begin{equation}
        m_q(\hr_r,\alpha_I) = (q-1)^{n(I)-1}\cdot q^{r-|I|-n(I)+1}\notag
    \end{equation}
    as claimed.
    \newline\noindent \textbf{Case 3:} Assume $1\in I$ and $r\notin I$. \\
    This implies that $\alpha_I$ contains the term $\alpha_1$, and the term with the largest index is $\alpha_{j_{n(I)}}$, with $j_{n(I)}<r$. Additionally, by \Cref{def: interval partition}, $n(I^c)=n(I)$ and
    \begin{equation}
        I^c = [j_{n(I)}+1,r-1] \sqcup \bigsqcup_{x=1}^{n(I)-1}[j_x+1,i_{x+1}-1].\label{interval partition for case 3}
    \end{equation}
    Therefore, in this case, we can state \Cref{big result starting point} as,
    \begin{align}
        m_q(\hr_r,\alpha_I)
        &= \left(\sum_{\substack{\sigma\in\mathcal{A}_{r}(\hr_r,\alpha_I) \\ \sigma = \sigma_{n(I)}}}(-1)^{\ell(\sigma_k)}\qpartition{\alpha_{j_{n(I)}+1,r}-\sum_{i\in\mathcal{I}_k}\alpha_i}\right)\notag\\
        &\hspace{3cm}\cdot \left(\prod_{x=1}^{n(I)-1}\sum_{\substack{\sigma\in\mathcal{A}_r(\hr_r,\alpha_I)\\ \sigma = \sigma_x}}(-1)^{\ell(\sigma_x)}\qpartition{\alpha_{j_x+1,i_{x+1}-1}-\sum_{i\in\mathcal{I}_x}\alpha_i}\right).\label{case 3 starting point}
    \end{align}
    By \Cref{THE BEST THEOREM EVER},
    \begin{align}
        m_q(\hr_r,\alpha_I)
        &= m_q(\hr_{r-j_{n(I)}+1},\alpha_1)\cdot \prod_{x=1}^{n(I)-1}m_q(\hr_{i_{x+1}-j_x+1},\alpha_1+\alpha_{i_{x+1}-j_x+1}),\notag
    \end{align}
    and by \Cref{thm: difference of q's} and \Cref{kim result},
    \begin{equation}
        m_q(\hr_r,\alpha_I) = q^{r-j_{n(I)}} \cdot \prod_{x=1}^{n(I)-1}q^{i_{x+1}-j_k-1}-q^{i_{x+1}-j_{n(I)}-2}.\label{Case 3 with best lemma used}
    \end{equation}
    We conclude this case by showing \Cref{Case 3 with best lemma used} is equivalent to \Cref{eq: KWMF final poly}:
    \begin{align}
        q^{r-j_{n(I)}} \cdot \prod_{x=1}^{n(I)-1}q^{i_{x+1}-j_k-1}-q^{i_{x+1}-j_{n(I)}-2}&=q^{r-j_{n(I)}} \cdot \prod_{x=1}^{n(I)-1}q^{i_{x+1}-j_{n(I)}-2}\cdot(q-1)\notag \\
        &= q^{r-j_{n(I)}}\cdot(q-1)^{n(I)-1}\cdot\prod_{x=1}^{n(I)-1}q^{i_{x+1}-j_x-2} \notag \\
        &=(q-1)^{n(I)-1} \cdot q^{r-j_{n(I)}+\sum_{x=1}^{n(I)-1}i_{x+1}-j_x-2} \notag\\
        &= (q-1)^{n(I)-1}\cdot q^{r-j_{n(I)}-n(I)+1+\sum_{x=1}^{n(I)-1}i_{x+1}-j_x-1}.\label{eq: case 3 so so close}
    \end{align}
    Since $r=|I|+|I^c|$, by \Cref{interval partition for case 3}, 
    \begin{equation}
        |I^c| = r-j_{n(I)} - \sum_{x=1}^{n(I)-1} i_{x+1}-j_x-1,
    \end{equation}
    which implies that
    \begin{equation}
        m_q(\hr_{r},\alpha_I) = (q-1)^{n(I)-1}\cdot q^{r-|I|-n(I)+1}\notag
    \end{equation}
    as claimed.
    \newline\noindent \textbf{Case 4:} Assume both $1,r\notin I$. \\
    This implies that in $\alpha_I$ the term with the smallest index is $\alpha_{i_1}$, with $i_1>1$, and the term with the largest index is $\alpha_{j_{n(I)}}$, with $j_{n(I)}<r$. Additionally, by \Cref{def: interval partition}, $n(I^c)=n(I)+1$ and 
    \begin{align}
        I^c = [1,i_1-1] \sqcup [j_{n(I)}+1,r-1] \sqcup \bigsqcup_{x=1}^{n(I)-1}[j_x+1,i_{x+1}-1].\label{interval partition for case 4}
    \end{align}
    Therefore, in this case, we can state \Cref{big result starting point} as,
    \begin{align}
        &m_q(\hr_r,\alpha_I)=\left(\sum_{\substack{\sigma\in\thealtset\\\sigma=\sigma_1}}(-1)^{\ell(\sigma_1)}\qpartition{\alpha_{1,i_1-1}-\sum_{i\in\mathcal{I}_1}\alpha_i}\right)\notag\\
        &\hspace{3cm}\cdot\left(\sum_{\substack{\sigma\in\mathcal{A}_{r}(\hr_r,\alpha_I) \\ \sigma = \sigma_{n(I)+1}}}(-1)^{\ell(\sigma_{n(I)+1})}\qpartition{\alpha_{j_{n(I)}+1,r}-\sum_{i\in\mathcal{I}_{n(I)}}\alpha_i}\right)\notag \\
        &\hspace{3cm}\cdot \left(\prod_{x=2}^{n(I)}\sum_{\substack{\sigma\in\mathcal{A}_r(\hr_r,\alpha_I)\\ \sigma = \sigma_x}}(-1)^{\ell(\sigma_x)}\qpartition{\alpha_{j_x+1,i_{x+1}-1}-\sum_{i\in\mathcal{I}_x}\alpha_i}\right).\label{case 4 starting point}
    \end{align}
    By \Cref{THE BEST THEOREM EVER},
    \begin{align}
        m_q(\hr_r,\alpha_I)=m_q(\hr_{i_1},\alpha_{i_1})\cdot m_q(\hr_{r-j_{n(I)}+1},\alpha_1)\cdot  \prod_{x=1}^{n(I)-1}m_q(\hr_{i_{x+1}-j_x+1},\alpha_1+\alpha_{i_{x+1}-j_x+1}),
    \end{align}
    so by \Cref{thm: difference of q's} and \Cref{kim result}, 
    \begin{equation}
        m_q(\hr_r,\alpha_I) = q^{i_1-1}\cdot q^{r-j_{n(I)}}\cdot \prod_{x=1}^{n(I)-1}q^{i_{x+1}-j_x-1}-q^{i_{x+1}-j_x-2}.\label{case 4 wrap it up}
    \end{equation}
    We conclude by showing \Cref{case 4 wrap it up} is equivalent to \Cref{eq: KWMF final poly}:
    \begin{align}
        q^{i_1-1}\cdot q^{r-j_{n(I)}} \cdot \prod_{x=1}^{n(I)-1}q^{i_{x+1}-j_x-1}-q^{i_{x+1}-j_x-2}
        &=q^{i_1-1}\cdot q^{r-j_{n(I)}} \cdot \prod_{x=1}^{n(I)-1}q^{i_{x+1}-j_x-2}\cdot (q-1)\notag\\
        &= (q-1)\cdot q^{i_1-1+r-j_{n(I)}+\sum_{x=1}^{n(I)-1}i_{x+1}-j_x-2} \notag\\
        &= (q-1)^{n(I)-1}\cdot q^{i_1-1+r-j_{n(I)}-n(I)+1+\sum_{x=1}^{n(I)-1}i_{x+1}-j_x-1}.\notag\label{eq: IT'S FINALLY OVER}
    \end{align}
    Since $r=|I|+|I^c|$, by \Cref{interval partition for case 4}
    \begin{equation}
        |I^c| = i_1-1+r-j_{n(I)} + \sum_{x=1}^{n(I)-1}i_{x+1}-j_x-1,
    \end{equation}
    which implies that
    \begin{equation}
        m_q(\hr_r,\alpha_I) = (q-1)^{n(I)-1}\cdot q^{r-|I|-n(I)+1}\notag
    \end{equation}
    as claimed, which completes the proof.
\end{proof}
\section*{Acknowledgements}
The author would like to thank Alex Wilson for sharing the problem and for their guidance in directions for future study. The author would also like to thank Kimberly J. Harry for sharing their guidance in the beginning drafts of this work. Lastly, the author would like to thank their advisor Pamela E. Harris for introducing the author to Kostant's Weight Multiplicity Formula.
\bibliographystyle{plain}
\bibliography{bibliography.bib}
\end{document}